\documentclass[]{amsart}
\pdfoutput=1
\usepackage[left=28truemm,right=28truemm,top=25truemm,bottom=25truemm]{geometry}
\usepackage[hidelinks,bookmarks=true,bookmarksnumbered=true,bookmarksopen=true,setpagesize=false]{hyperref}
\usepackage{caption}
\usepackage{amsmath,amssymb}
\usepackage{amsthm}
\usepackage{mathtools}
\mathtoolsset{showonlyrefs}
\theoremstyle{definition}
\usepackage{graphicx}
\usepackage{float}
\usepackage{bsvector}
\usepackage{tikz}
\usetikzlibrary{arrows.meta}
\newtheorem{definition}{Definition}
\newcommand{\definitionlab}[1]{\label{definition:#1}}
\newcommand{\definitionref}[1]{Definition~\ref{definition:#1}}
\newtheorem{proposition}{Proposition}
\newcommand{\propositionlab}[1]{\label{proposition:#1}}
\newcommand{\propositionref}[1]{Proposition~\ref{proposition:#1}} 
\newtheorem{theorem}{Theorem}
\newcommand{\theoremlab}[1]{\label{theorem:#1}}
\newcommand{\theoremref}[1]{Theorem~\ref{theorem:#1}}
\newtheorem{corollary}{Corollary}
\newcommand{\corollarylab}[1]{\label{corollary:#1}}
\newcommand{\corollaryref}[1]{Corollary~\ref{corollary:#1}}
\newtheorem{lemma}{Lemma}
\newcommand{\lemmalab}[1]{\label{lemma:#1}}
\newcommand{\lemmaref}[1]{Lemma~\ref{lemma:#1}}

\newtheorem{remark}{Remark}
\newcommand{\remarklab}[1]{\label{remark:#1}}
\newcommand{\remarkref}[1]{Remark~\ref{remark:#1}}
\newtheorem{example}{Example}
\newcommand{\examplelab}[1]{\label{example:#1}}
\newcommand{\exampleref}[1]{Example~\ref{example:#1}}

\newcommand{\LL}[1]{{L}_{#1}}
\newcommand{\hh}[2]{\widetilde{h}(#1,#2)}
\newcommand{\F}[1]{F(#1)}

\newcommand{\FF}[2]{\widetilde{F}(#1,#2)}
\newcommand{\FFO}[2]{f(#1)+\hh{#1}{#2}}
\newcommand{\FFxgrad}[2]{\nabla_\bsx\FF{#1}{#2}}
\newcommand{\FFxgradO}[2]{\nabla_\bsx(f(#1)+\hh{#1}{#2})}
\renewcommand{\email}[1]{\texttt{(E-mail:~#1)}}

\title{Convergence Rate Analysis of Continuous- and Discrete-Time Smoothing Gradient Algorithms}
\author[M.~Toyoda]{Mitsuru Toyoda${}^1$}
\author[A.~Nishioka]{Akatsuki Nishioka${}^2$}
\author[M.~Tanaka]{Mirai Tanaka${}^3$}

\thanks{
	${}^1$Department of Mechanical Systems Engineering, 
	Tokyo Metropolitan University, 
	6-6 Asahigaoka, 
	Hino-shi, 
	191-0065 Tokyo, Japan \email{toyoda@tmu.ac.jp}} %
\thanks{
	${}^2$Department of Mathematical Informatics, 
	Graduate School of Information Science and Technology, 
	The University of Tokyo %
	\email{akatsuki\_nishioka@mist.i.u-tokyo.ac.jp}}%
\thanks{
	${}^3$Department of Statistical Inference and Mathematics, 
	The Institute of Statistical Mathematics, 
	10-3 Midori-cho, 
	Tachikawa-shi, 
	190-8562 Tokyo, Japan, and 
	Continuous Optimization Team, 
	RIKEN Center for Advanced Intelligence Project, 
	Nihonbashi 1-chome Mitsui Building, 15th floor, 1-4-1 Nihonbashi, 
	Chuo-ku, 
	103-0027 Tokyo, Japan \email{mirai@ism.ac.jp}}

\begin{document}
\begin{abstract}
	This paper addresses the gradient flow--the continuous-time representation of the gradient method--with the smooth approximation of a non-differentiable objective function
	and presents convergence analysis framework.
	Similar to the gradient method, the gradient flow is inapplicable to the non-differentiable function minimization;
	therefore, this paper addresses the smoothing gradient method,
	which exploits a decreasing smoothing parameter sequence in the smooth approximation.
	The convergence analysis is presented using conventional Lyapunov-function-based techniques,
	and a Lyapunov function applicable to both strongly convex and non-strongly convex objective functions is provided
	by taking into consideration the effect of the smooth approximation.
	Based on the equivalence of the stepsize in the smoothing gradient method and
	the discretization step in the forward Euler scheme for
	the numerical integration of the smoothing gradient flow,
	the sample values of the exact solution of the smoothing gradient flow are compared with
	the state variable of the smoothing gradient method, and
	the equivalence of the convergence rates is shown.
\end{abstract}

\maketitle

\section{Introduction}
The smoothing method has been exploited to replace a non-differentiable function with an alternative smooth function and make the gradient-based algorithms applicable.
Besides, the proximal algorithms for a class of the non-differentiable objective functions having friendly structures to the proximal operation (see, e.g., \cite{Pa14})
have recently attracted much attention; however, the applicable problems are limited, and
the development and analysis of the smoothing method are still required.

The smoothing algorithms are broadly separated into two classes: the smoothing parameter is fixed or varying.
The fixed smoothing parameter (see, e.g., Section 10.8 of \cite{Be17a}) focuses on obtaining an $\varepsilon$-optimal solution
and does not have the theoretical guarantee of the convergence to the optimal value.
On the other hand, the convergence to the optimal value has been established
using decreasing smoothing parameter in existing reports \cite{Ch12b}.
The detailed convergence analysis taking into consideration the design of the smoothing parameter sequence is necessary \cite{Bi20,Wu23}
because the decreasing rate of the smoothing parameter sequence has a huge effect on the convergence performance of the algorithms.

Besides the aforementioned algorithms are given by discrete-time difference equations,
the gradient flow, which is a representation of the gradient method as a continuous-time ordinary differential equation (ODE),
has been explored in the research fields of the differential equations and related numerical analysis \cite{Br89}.
For an ODE representation \cite{Su16} of the Nesterov's accelerated gradient method \cite{Ne14},
a smoothing approximation for the ODE and related convergence analysis are proposed \cite{Qu22}.
As points out in a research \cite{Us22},
the transform of the timescale and state variables leads to an arbitrary convergence rate;
therefore, unified convergence rate evaluation in the continuous-time systems is not clear.
In the practical numerical computation,
the gradient flow, given by an ODE, generally does not have an analytic solution and
requires to be numerically integrated; therefore,
the comparisons in terms of some practical timescale equivalent to the actual computation time, such as the evaluation time of the gradient of the objective function,
is preferred in the evaluation of the algorithms.
The resulting discrete-time difference equations depend on many kinds of the discretization schemes \cite{Do80,Pr92} and related stepsize design,
and the unified comparisons on the convergence rates in terms of the resulting discrete step is still not clear.

Based on the aforementioned background, this paper focuses on the smoothing gradient flow (continuous-time system)
and an equivalent gradient method (discrete-time system) obtained using the forward Euler scheme,
and the convergence rates of the systems are examined. The main contribution of this paper is summarized as follows:
\begin{enumerate}
	\item Similar Lyapunov functions for the discrete-time and continuous-time systems are introduced and
	      exploited to derive convergence rates. The proposed Lyapunov functions are
	      extension of the conventional Lyapunov functions \cite{Wi18}
	      and take into consideration the smoothing approximation and the strong convexity.
	\item For the stepsize design in the discrete-time systems,
	      the relationship to the continuous timeline in the continuous-time system is explored,
	      and it is pointed out that the discrete-time and continuous-time systems have the difference expression of the convergence rate.
	\item A novel stepsize and smoothing parameter sequence design method is proposed based on
	      the aforementioned continuous-time system analysis, and associated convergence guarantee is provided.
\end{enumerate}

\section{Notations}
\begin{itemize}
	\item For a sequence $a_k\,(k \geq 0)$, a sum $\sum_{k=0}^{N}a_k = 0$ is defined for $N < 0$.
	      Furthermore, a product $\prod_{k=0}^{N}a_k = 1$ is defined for $N < 0$.
	      There formal definitions are introduced so that
	      an equations $\sum_{k=0}^{N}a_k - \sum_{k=0}^{N-1}a_k = a_N$ and $\prod_{k=0}^{N}a_k = a_N\prod_{k=0}^{N-1}a_k$
	      hold for $N \leq 0$.
	\item For functions $f(t)$ and $g(t)$ on $[t_0,+\infty)$,
	      if there exist $C>0$ and $T\in[t_0,+\infty)$ such that $f(t) \leq Cg(t)$ ($f(t) \geq Cg(t)$) for arbitrary $t \geq T$,
	      a notation $f(t) = \mathcal{O}(g(t))$ ($f(t) = \Omega(g(t))$) is used.
	      If $f(t) = \mathcal{O}(g(t))$ and $f(t) = \Omega(g(t))$,
	      a notation $f(t) = \Theta(g(t))$ is used.
	\item Similarly, for sequences $a_k$ and $b_k$ on $k=0,1,\dots$,
	      if there exist $C>0$ and $K \geq 0$ such that
	      $a_k \leq Cb_k$ ($a_k \geq Cb_k$) for $k \geq K$,
	      a notation $a_k = \mathcal{O}(b_k)$ ($a_k = \Omega(b_k)$) is used.
	      If $a_k = \mathcal{O}(b_k)$ and $a_k = \Omega(b_k)$,
	      a notation $a_k = \Theta(b_k)$ is used.
\end{itemize}
\section{Problem Formulation}
A smooth approximation of a possibly non-differentiable function is defined as follows (Definition 10.43 of \cite{Be17a}):
\begin{definition}\definitionlab{1}
	A convex function $h:\R^{n_\bsx}\to\R$ is considered.
	A convex differentiable function $\hh{\cdot}{\mu}:\R^{n_\bsx}\to\R$ satisfying the conditions below
	for an arbitrary $\mu>0$ is called a $(1/\mu)$-smooth approximation of $h$ with parameters $(\alpha,\beta)$:
	\begin{enumerate}
		\item For an arbitrary $\bsx\in\R^{n_\bsx}$,
		      an inequality $\hh{\bsx}{\mu} \leq h(\bsx) \leq \hh{\bsx}{\mu} + \beta \mu$ holds $(\beta > 0)$.
		\item $\hh{\cdot}{\mu}$ is $(\alpha/\mu)$-smooth $(\alpha > 0)$.
	\end{enumerate}
\end{definition}
Subsequently, the Lipschitz continuous smooth approximation, meaning
a smooth approximation that is Lipschitz continuous with respect to $\mu$,
is introduced.
\begin{definition}\definitionlab{3}
	A $(1/\mu)$-smooth approximation of $h$ with parameters $(\alpha,\beta)$
	is considered. For an arbitrary $\bsx\in\R^{n_\bsx}$ and $\mu>0$, if an inequality
	\begin{equation}
		-\beta \leq \nabla_{\mu}\hh{\bsx}{\mu} \leq 0
		\elab{Definition2}
	\end{equation}
	holds, $h$ is called
	a $(1/\mu)$-Lipschitz continuous smooth approximation of $h$ with parameters $(\alpha,\beta)$.
\end{definition}
\begin{remark}
	$\hh{\bsx}{\mu} \leq h(\bsx) \leq \hh{\bsx}{\mu} + \beta \mu$ of \definitionref{1}
	and $-\beta \leq \nabla_{\mu}\hh{\bsx}{\mu} \leq 0$ of \definitionref{3} are
	somewhat strong compared with those of existing reports \cite{Qu22,Bi20},
	given by $-\beta \leq \nabla_{\mu}\hh{\bsx}{\mu} \leq \beta$.
	Definitions \ref{definition:1} and \ref{definition:3} of this paper halve the two smooth approximation error terms in \cite{Qu22,Bi20}.
\end{remark}
\begin{remark}
	The integration of the both side of the inequality $-\beta \leq \nabla_{\mu}\hh{\bsx}{\mu} \leq 0$ from $\mu = 0+$ to $\mu = \nu$
	results in an inequality $-\beta\nu \leq \hh{\bsx}{\nu} - h(\bsx) \leq 0$.
	On replacing $\nu$ with $\mu$ and rearranging,
	an inequality $\hh{\bsx}{\mu} \leq h(\bsx) \leq \hh{\bsx}{\mu} + \beta \mu$ is obtained and
	is consistent with the condition for fixed smoothing parameters (\definitionref{1}).
\end{remark}
The condition of the Lipschitz continuous approximation holds in common smooth approximations,
such as square-root-based and
Huber's $\ell_2$-norm approximations
and the log-sum-exp approximation, as presented in the following examples.
\begin{example}\examplelab{l2NormSmoothingFunc1}(Example 10.44 of \cite{Be17a})
	A $(1/\mu)$-smooth approximation of the $\ell_2$-norm with parameters $(1,1)$
	and its gradient with respect to $\bsx$ are given as follows:
	\begin{equation*}
		\widetilde{h}_{\ell_2}(\bsx,\mu) = \sqrt{\|\bsx\|_2^2 + \mu^2} - \mu,
		\,\,
		\nabla_{\bsx}\widetilde{h}_{\ell_2}(\bsx,\mu)
		= \frac{\bsx}{\sqrt{\|\bsx\|_2^2 + \mu^2}}.
	\end{equation*}
	The partial derivative of $\widetilde{h}_{\ell_2}(\bsx,\mu)$ with respect to $\mu > 0$ is
	\begin{equation*}
		\nabla_\mu\widetilde{h}_{\ell_2}(\bsx,\mu) = \frac{\mu}{\sqrt{\|\bsx\|_2^2 + \mu^2}} -1
		\in[-1, 0),
	\end{equation*}
	which means that $-1 \leq \nabla_\mu\widetilde{h}_{\ell_2}(\bsx,\mu) < 0$.
\end{example}
\begin{example}\examplelab{l2NormSmoothingFunc2}(Example 10.53 of \cite{Be17a})
	A $(1/\mu)$-smooth approximation of the $\ell_2$-norm with parameters $(1,1/2)$,
	named the Huber function, and its gradient with respect to $\bsx$ are given as follows:
	\begin{equation*}
		\widetilde{h}_{\ell_2}(\bsx,\mu) =
		\begin{cases}{}
			\dfrac{1}{2\mu}\|\bsx\|_2^2 & (\|\bsx\|_2 \leq \mu), \\
			\|\bsx\|_2 - \dfrac{\mu}{2} & (\|\bsx\|_2 > \mu),
		\end{cases}\,\,
		\nabla_\bsx\widetilde{h}_{\ell_2}(\bsx,\mu) =
		\begin{cases}{}
			\dfrac{\bsx}{\mu}        & (\|\bsx\|_2 \leq \mu), \\
			\dfrac{\bsx}{\|\bsx\|_2} & (\|\bsx\|_2 > \mu).
		\end{cases}
	\end{equation*}
	Furthermore, the partial derivative with respect to $\mu$ is calculated as follows:
	\begin{equation*}
		\nabla_\mu\widetilde{h}_{\ell_2}(\bsx,\mu) =
		\begin{cases}{} \displaystyle\nabla_\mu\left(\frac{1}{2\mu}\|\bsx\|_2^2 \right) = - \frac{1}{2\mu^2}\|\bsx\|_2^2 \in \left[-\frac{1}{2},0\right] & (\|\bsx\|_2 \leq \mu), \\
             \displaystyle \nabla_\mu\left(\|\bsx\|_2 - \frac{1}{2\mu} \right) = -\frac{1}{2}                                                    & (\|\bsx\|_2 > \mu).
		\end{cases}
	\end{equation*}
	Consequently, an inequality
	$-1/2 \leq \nabla_\mu\widetilde{h}_{\ell_2}(\bsx,\mu) \leq 0$ is obtained.
\end{example}

\begin{example}\examplelab{LogSumExp}(Example 10.45 of \cite{Be17a})
	An operation $\max(\bsx)$, which calculates the maximum element of a vector $\bsx\in\R^{n_\bsx}$,
	is taken into consideration.
	A $(1/\mu)$-smooth approximation of $\max(\bsx)$ with parameters $(1,\log n_\bsx)$
	is given by
	\begin{equation*}
		\widetilde{h}_{\max}(\bsx,\mu) = \mu \log\left(\sum_{i=1}^{n_\bsx}\ee^{x_i/\mu} \right) -
		\mu \log n_\bsx.
	\end{equation*} The partial derivative with respect to $\mu$ is
	\begin{equation*}
		\nabla_\mu\widetilde{h}_{\max}(\bsx,\mu)
		= \log\left(\sum_{i=1}^{n_\bsx}\ee^{x_i/\mu} \right) -
		\left.
		\dsum_{i=1}^{n_\bsx}\ee^{x_i/\mu}\dfrac{x_i}{\mu}
		\middle/
		\dsum_{i=1}^{n_\bsx}\ee^{x_i/\mu}
		\right.
		- \log n_\bsx.
	\end{equation*}
	On focusing on the inside of the sum above,
	variable transform
	$z_i = \ee^{x_i/\mu} > 0$, meaning $\log z_i = x_i / \mu$, is introduced.
	The equation above is rearranged as
	\begin{equation*}
		\nabla_\mu\widetilde{h}_{\max}(\bsx,\mu)
		= \log\left(\sum_{i=1}^{n_\bsx}z_i \right)
		- \left.
		\dsum_{i=1}^{n_\bsx} z_i \log z_i
		\middle/
		\dsum_{i=1}^{n_\bsx}z_i
		\right.
		- \log n_\bsx.
	\end{equation*}
	On recalling a function $-\log$ is convex
	and using a vector $\bsz = [z_1,\dots,z_{n_\bsx}]^\top\in\R_{>0}^{n_\bsx}$,
	Jensen's inequality indicates
	\begin{equation*}
		- \left.
		\dsum_{i=1}^{n_\bsx} z_i \log z_i
		\middle/
		\dsum_{i=1}^{n_\bsx}z_i
		\right.
		\geq - \log\left(
		\dsum_{i=1}^{n_\bsx} z_i^2
		\middle/
		\dsum_{i=1}^{n_\bsx}z_i
		\right)
		= - \log\frac{\|\bsz\|_2^2}{\|\bsz\|_1}.
	\end{equation*}
	The inequality above and a basic inequality of the $\ell_2$- and $\ell_1$-norms
	given by $\|\bsz\|_2 \leq \|\bsz\|_1$ result in the following inequality:
	\begin{equation*}
		\nabla_\mu\widetilde{h}_{\max}(\bsx)
		\geq \log \|\bsz\|_1
		- \log\frac{\|\bsz\|_2^2}{\|\bsz\|_1} - \log n_\bsx
		= - \log\frac{\|\bsz\|_2^2}{\|\bsz\|_1^2} - \log n_\bsx
		\geq - \log n_\bsx.
	\end{equation*}
\end{example}

The following theorem is used to
calculate the Lipschitz-smoothness parameter of
a weighted sum of two smooth approximations with affine scaling.
\begin{theorem}\theoremlab{LipschitzSum}
	A $(1/\mu)$-Lipschitz continuous smooth approximation $\widetilde{h}_1(\bsx_1,\mu)$
	of $h_1(\bsx_1)$ with parameters $(\alpha_1,\beta_1)$
	and a $(1/\mu)$-Lipschitz continuous smooth approximation $\widetilde{h}_2(\bsx_2,\mu)$
	of $h_2(\bsx_2)$ with parameters $(\alpha_2,\beta_2)$ are considered.
	For weight coefficients $w_1 \geq 0$ and $w_2 \geq 0$,
	and matrices and vectors $\bsA_1$, $\bsA_2$, $\bsb_1$, and $\bsb_2$ with
	appropriate sizes,
	a function $\widetilde{h}_\mathrm{sum}(\bsxi,\mu) = w_1 \widetilde{h}_1(\bsA_1\bsxi+\bsb_1,\mu) + w_2 \widetilde{h}_2(\bsA_2\bsxi+\bsb_2,\mu)$
	is
	a $(w_1\alpha_1\|\bsA_1\|_2^2 + w_2\alpha_2\|\bsA_2\|_2^2,
		w_1\beta_1 + w_2\beta_2)$-Lipschitz continuous smooth approximation function
	of $h_\text{sum}(\bsxi) = w_1 h_1(\bsA_1\bsxi+\bsb_1) + w_2 h_2(\bsA_2\bsxi+\bsb_2)$.
\end{theorem}
\begin{proof}
	Based on Theorem~10.46 of \cite{Be17a},
	$\widetilde{h}_\text{sum}$ is
	a $(1/\mu_k)$-smooth approximation function of $h_\text{sum}$
	with parameters $(w_1\alpha_1\|\bsA_1\|_2^2 + w_2\alpha_2\|\bsA_2\|_2^2, w_1\beta_1 + w_2\beta_2)$.
	The inequality \eref{Definition2}, required by the Lipschitz continuous smoothability,
	is obtained on combining the inequalities \eref{Definition2} for $h_1$ and $h_2$
	results in $-\beta_1 \leq \nabla_\mu \widetilde{h}_1(\bsA_1\bsxi+\bsb_1,\mu) \leq 0$
	and $-\beta_2 \leq \nabla_\mu \widetilde{h}_2(\bsA_2\bsxi+\bsb_2,\mu) \leq 0$.
	The combination of the two inequalities results in
	$-(w_1\beta_1 + w_2\beta_2) \leq \nabla_\mu
		[w_1\widetilde{h}_1(\bsA_1\bsxi+\bsb_1,\mu) + w_2\widetilde{h}_2(\bsA_2\bsxi+\bsb_2,\mu)]
		\leq 0$, and the claim of the theorem is obtained.
\end{proof}

This paper addresses the following optimization problem:
\begin{equation}
	\minimize_{\bsx\in\R^{n_\bsx}}\,\, F(\bsx) := f(\bsx) + h(\bsx).
	\elab{Prob}
\end{equation}
In the problem above, the following assumptions are placed (see Assumption 10.56 of \cite{Be17a}):
\begin{itemize}
	\item $f:\R^{n_\bsx}\to\R$ is $\sigma$-strongly convex $(\sigma \geq 0)$ and $L$-smooth $(L \geq 0)$.
	\item $h:\R^{n_\bsx}\to\R$ is convex, and
	      a $(1/\mu)$-smooth approximation $\widetilde{h}(\cdot,\mu)$ of $h$ with parameters $(\alpha,\beta)$ is given.
	\item An optimal solution $\bsx^*$ exists.
\end{itemize}
A simplified notation $\FF{\bsx}{\mu} := \FFO{\bsx}{\mu}$ is used throughout this paper.
The smoothing parameter sequence $\mu(t)$ for smoothing $h$ is assumed to be non-increasing and converge to the zero
(i.e., $\mu(t)>0\,(t\in[t_0,+\infty),\,\dot{\mu}(t)\leq 0,\,\mu(t)\searrow 0)$) hereafter.
$\FF{\cdot}{\mu(t)} = f(\cdot) + \hh{\cdot}{\mu(t)}$ is $\sigma$-strongly convex and $L(t)$-smooth $(L(t):=L + \alpha/\mu(t))$, meaning differentiable,
and the gradient flow of the smooth approximation $\FF{\cdot}{\mu(t)} = f(\cdot) + \hh{\cdot}{\mu(t)}$ is given by
\begin{equation}
	\begin{split}
		\dot{\bsx}(t)
		 & = -\FFxgrad{\bsx(t)}{\mu(t)}                  \\
		 & = -\FFxgradO{\bsx(t)}{\mu(t)} \, (t\geq t_0),
	\end{split}
	\elab{SmoothingGradientFlow}
\end{equation}
and the ODE above is called the smoothing gradient flow of $F = f + h$,
which is referred to as ``the continuous-time system'' in this paper.
In addition, the discretization with a stepsize $s_k\,(k=0,1,\dots,)$ satisfying $t_{k+1} = t_{k} + s_k$,
which is the well-known forward Euler scheme, results in the following discrete-time difference equation:
\begin{equation}
	\begin{split}
		\bsxI{k+1} & = \bsxI{k} - s_k\FFxgrad{\bsxI{k}}{\mu_k}               \\
		           & = \bsxI{k} - s_k\FFxgradO{\bsxI{k}}{\mu_k}\,(k \geq 0),
	\end{split}
	\elab{SmoothingGradientFlowDiscrete}
\end{equation}
where the correspondence of the discrete- and continuous-time state variables is as follows:
\begin{equation}
	\left\{
	\begin{array}{l}
		\bsx(t_k)\approx\bsxI{k},                                                                    \\
		\dot{\bsx}(t_k) \approx (\bsx(t_{k+1}) - \bsx(t_k))/s_k \approx (\bsxI{k+1} - \bsxI{k})/s_k, \\
		\mu(t_k) = \mu_k\,(k\geq 0,\,\mu_{k+1} - \mu_{k} \leq 0, \mu_k \searrow 0).
	\end{array}
	\right.
\end{equation}
The aforementioned variables are illustrated in \fref{sample}.
Besides the analytical solution of the continuous-time system (i.e., the explicit expression of $\bsx(t)$) is generally unavailable,
the values $\bsx(t_k) \, (k=0,\dots,)$ are called
the sample values of the continuous-time system in this paper and compared with
the discrete-time system in the subsequent theoretical analysis.
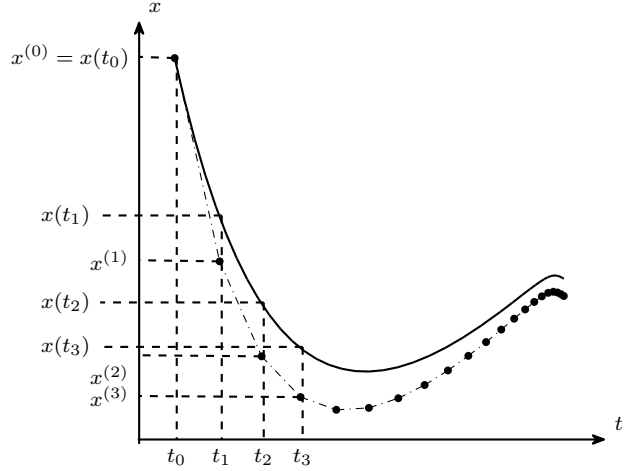
\begin{figure}[H]
	\centering
	\begin{tikzpicture}[xscale = 50, yscale = 50]
		\draw[->, >={Stealth[round]}, thick] (0.9896486534300927,0.5)--(1.113864812268981,0.5) node [above right] {\footnotesize $t$};
		\draw[->, >={Stealth[round]}, thick] (0.99,0.49880633201568786)--(0.99,0.6111873334049689) node [above right] {\footnotesize $x$};
		\coordinate (Euler0) at (1.0,0.6018222499558621);
		\coordinate (Euler1) at (1.0119620664009612,0.5477064861421153);
		\coordinate (Euler2) at (1.0231318337668756,0.5224157553148007);
		\coordinate (Euler3) at (1.0334838987871433,0.5115621590529748);
		\coordinate (Euler4) at (1.043000228160828,0.5081714154647946);
		\coordinate (Euler5) at (1.051671921525995,0.5086986785825993);
		\coordinate (Euler6) at (1.059500730175761,0.5112535313886795);
		\coordinate (Euler7) at (1.0665001077193657,0.5147806121579791);
		\coordinate (Euler8) at (1.0726955835543717,0.5186677522730201);
		\coordinate (Euler9) at (1.0781243155244704,0.5225513507773234);
		\coordinate (Euler10) at (1.0828337901933767,0.5262152016781014);
		\coordinate (Euler11) at (1.0868797751353967,0.5295344954606065);
		\coordinate (Euler12) at (1.0903237517185629,0.532441792415925);
		\coordinate (Euler13) at (1.093230133515338,0.534903183484972);
		\coordinate (Euler14) at (1.0956635846386136,0.5368975711325026);
		\coordinate (Euler15) at (1.097686697380768,0.5383933332941551);
		\coordinate (Euler16) at (1.0993581928883192,0.539323496365365);
		\coordinate (Euler17) at (1.100731703457039,0.539619077010665);
		\coordinate (Euler18) at (1.1018551065864601,0.5394017160355853);
		\coordinate (Euler19) at (1.1027703230873371,0.5389579501391905);
		\coordinate (Euler20) at (1.1035134656990735,0.5384765649673918);
		\fill (Euler0) circle [radius=0.03pt];
		\fill (Euler1) circle [radius=0.03pt];
		\fill (Euler2) circle [radius=0.03pt];
		\fill (Euler3) circle [radius=0.03pt];
		\fill (Euler4) circle [radius=0.03pt];
		\fill (Euler5) circle [radius=0.03pt];
		\fill (Euler6) circle [radius=0.03pt];
		\fill (Euler7) circle [radius=0.03pt];
		\fill (Euler8) circle [radius=0.03pt];
		\fill (Euler9) circle [radius=0.03pt];
		\fill (Euler10) circle [radius=0.03pt];
		\fill (Euler11) circle [radius=0.03pt];
		\fill (Euler12) circle [radius=0.03pt];
		\fill (Euler13) circle [radius=0.03pt];
		\fill (Euler14) circle [radius=0.03pt];
		\fill (Euler15) circle [radius=0.03pt];
		\fill (Euler16) circle [radius=0.03pt];
		\fill (Euler17) circle [radius=0.03pt];
		\fill (Euler18) circle [radius=0.03pt];
		\fill (Euler19) circle [radius=0.03pt];
		\fill (Euler20) circle [radius=0.03pt];
		\draw[dash dot] (Euler0)--(Euler1);
		\draw[dash dot] (Euler1)--(Euler2);
		\draw[dash dot] (Euler2)--(Euler3);
		\draw[dash dot] (Euler3)--(Euler4);
		\draw[dash dot] (Euler4)--(Euler5);
		\draw[dash dot] (Euler5)--(Euler6);
		\draw[dash dot] (Euler6)--(Euler7);
		\draw[dash dot] (Euler7)--(Euler8);
		\draw[dash dot] (Euler8)--(Euler9);
		\draw[dash dot] (Euler9)--(Euler10);
		\draw[dash dot] (Euler10)--(Euler11);
		\draw[dash dot] (Euler11)--(Euler12);
		\draw[dash dot] (Euler12)--(Euler13);
		\draw[dash dot] (Euler13)--(Euler14);
		\draw[dash dot] (Euler14)--(Euler15);
		\draw[dash dot] (Euler15)--(Euler16);
		\draw[dash dot] (Euler16)--(Euler17);
		\draw[dash dot] (Euler17)--(Euler18);
		\draw[dash dot] (Euler18)--(Euler19);
		\draw[dash dot] (Euler19)--(Euler20);
		\draw[dashed, thick] (Euler0)--(0.99,0.6018222499558621) node [left] {\footnotesize $x^{(0)}=x(t_0)$};
		\draw[dashed, thick] (Euler1)--(0.99,0.5477064861421153) node [left] {\footnotesize $x^{(1)}$};
		\draw[dashed, thick] (Euler2)--(0.99,0.5224157553148007) node [below left] {\footnotesize $x^{(2)}$};
		\draw[dashed, thick] (Euler3)--(0.99,0.5115621590529748) node [left] {\footnotesize $x^{(3)}$};
		\coordinate (ODE0) at (1.0,0.6018222499558621);
		\coordinate (ODE1) at (1.0026097222922346,0.590639799941053);
		\coordinate (ODE2) at (1.0053762078606734,0.5800649481043978);
		\coordinate (ODE3) at (1.0081505385749356,0.5706616559834858);
		\coordinate (ODE4) at (1.0109360569309962,0.5623186136667971);
		\coordinate (ODE5) at (1.0137325954780771,0.5549468167088453);
		\coordinate (ODE6) at (1.016519520327677,0.5485088597436256);
		\coordinate (ODE7) at (1.019318702967351,0.5428714026775002);
		\coordinate (ODE8) at (1.0222115362620594,0.5378373059428238);
		\coordinate (ODE9) at (1.0251607523500161,0.5334561404604665);
		\coordinate (ODE10) at (1.028168820954237,0.5296943318748967);
		\coordinate (ODE11) at (1.0312475197633748,0.5265120860753396);
		\coordinate (ODE12) at (1.0344053149370045,0.5238814685483358);
		\coordinate (ODE13) at (1.0376502623307575,0.5217802894425112);
		\coordinate (ODE14) at (1.0409905109536524,0.5201907018680422);
		\coordinate (ODE15) at (1.0444341833638329,0.5190986158261757);
		\coordinate (ODE16) at (1.0479893414343082,0.5184930871579354);
		\coordinate (ODE17) at (1.0516639863528607,0.518365770510843);
		\coordinate (ODE18) at (1.0554660892693157,0.5187104518169342);
		\coordinate (ODE19) at (1.0594036769057156,0.5195226630039654);
		\coordinate (ODE20) at (1.0634850149698436,0.5207993853126323);
		\coordinate (ODE21) at (1.067719137207035,0.5225389266516702);
		\coordinate (ODE22) at (1.0721172589770567,0.5247412838210851);
		\coordinate (ODE23) at (1.0766972774725778,0.5274103747185289);
		\coordinate (ODE24) at (1.0814999617086694,0.5305642622760283);
		\coordinate (ODE25) at (1.0866487996483813,0.534277111417341);
		\coordinate (ODE26) at (1.089867109050898,0.5367269421773271);
		\coordinate (ODE27) at (1.0924585280252805,0.5387359122028959);
		\coordinate (ODE28) at (1.094428770170441,0.5402541480462564);
		\coordinate (ODE29) at (1.0959862461308916,0.5414179953648103);
		\coordinate (ODE30) at (1.0972373804743645,0.5422981401849677);
		\coordinate (ODE31) at (1.0982764361409512,0.542958333105176);
		\coordinate (ODE32) at (1.0991769694920746,0.5434414743647271);
		\coordinate (ODE33) at (1.100077502843198,0.5437922192723997);
		\coordinate (ODE34) at (1.1009085057266885,0.543944428126566);
		\coordinate (ODE35) at (1.1017739278203555,0.5438801142450597);
		\coordinate (ODE36) at (1.102470895279489,0.5436542165906748);
		\coordinate (ODE37) at (1.1031678627386226,0.5432900640157271);
		\coordinate (ODE38) at (1.1035134656990735,0.5430671009257023);
		\draw[thick] (ODE0)--(ODE1);
		\draw[thick] (ODE1)--(ODE2);
		\draw[thick] (ODE2)--(ODE3);
		\draw[thick] (ODE3)--(ODE4);
		\draw[thick] (ODE4)--(ODE5);
		\draw[thick] (ODE5)--(ODE6);
		\draw[thick] (ODE6)--(ODE7);
		\draw[thick] (ODE7)--(ODE8);
		\draw[thick] (ODE8)--(ODE9);
		\draw[thick] (ODE9)--(ODE10);
		\draw[thick] (ODE10)--(ODE11);
		\draw[thick] (ODE11)--(ODE12);
		\draw[thick] (ODE12)--(ODE13);
		\draw[thick] (ODE13)--(ODE14);
		\draw[thick] (ODE14)--(ODE15);
		\draw[thick] (ODE15)--(ODE16);
		\draw[thick] (ODE16)--(ODE17);
		\draw[thick] (ODE17)--(ODE18);
		\draw[thick] (ODE18)--(ODE19);
		\draw[thick] (ODE19)--(ODE20);
		\draw[thick] (ODE20)--(ODE21);
		\draw[thick] (ODE21)--(ODE22);
		\draw[thick] (ODE22)--(ODE23);
		\draw[thick] (ODE23)--(ODE24);
		\draw[thick] (ODE24)--(ODE25);
		\draw[thick] (ODE25)--(ODE26);
		\draw[thick] (ODE26)--(ODE27);
		\draw[thick] (ODE27)--(ODE28);
		\draw[thick] (ODE28)--(ODE29);
		\draw[thick] (ODE29)--(ODE30);
		\draw[thick] (ODE30)--(ODE31);
		\draw[thick] (ODE31)--(ODE32);
		\draw[thick] (ODE32)--(ODE33);
		\draw[thick] (ODE33)--(ODE34);
		\draw[thick] (ODE34)--(ODE35);
		\draw[thick] (ODE35)--(ODE36);
		\draw[thick] (ODE36)--(ODE37);
		\draw[thick] (ODE37)--(ODE38);
		\draw[dashed, thick] (1.0,0.6018222499558621)--(1.0,0.5) node [below] {\footnotesize $t_{0}$};
		\draw[dashed, thick] (1.0119620664009612,0.5596140081832028)--(1.0119620664009612,0.5) node [below] {\footnotesize $t_{1}$};
		\draw[dashed, thick] (1.0231318337668756,0.536470171239448)--(1.0231318337668756,0.5) node [below] {\footnotesize $t_{2}$};
		\draw[dashed, thick] (1.0334838987871433,0.5246490590158022)--(1.0334838987871433,0.5) node [below] {\footnotesize $t_{3}$};
		\draw[dashed, thick] (1.0119620664009612,0.5596140081832028)--(0.9796486534300927,0.5596140081832028) node [left] {\footnotesize $x(t_1)$};
		\draw[dashed, thick] (1.0231318337668756,0.536470171239448)--(0.9796486534300927,0.536470171239448) node [left] {\footnotesize $x(t_2)$};
		\draw[dashed, thick] (1.0334838987871433,0.5246490590158022)--(0.9796486534300927,0.5246490590158022) node [left] {\footnotesize $x(t_3)$};
	\end{tikzpicture}
	\caption{Sample values of the continuous-time system $\bsx(t_k)$ and
		the discrete-time state variable $\bsxI{k}$.}
	\flab{sample}
\end{figure}

The difference equation \eref{SmoothingGradientFlowDiscrete} is called ``the discrete-time system'' in this paper
and indeed the smoothing gradient method for the smooth approximation
$\FF{\cdot}{\mu_k} = \FFO{\cdot}{\mu_k}$, which is $\sigma$-strongly convex and $L_k$-smooth $(L_k:=L+\alpha/\mu_k)$.

\section{Main Results: Convergence Analysis}
This section presents the convergence analysis of the continuous- and discrete-time systems obtained using the forward Euler scheme
with the consideration of the smooth approximation
(i.e., the smoothing gradient flow and smoothing gradient method, respectively);
in both the systems, the Lyapunov-function-based proof is given.
Besides the Lyapunov function for strongly and non-strongly convex objective functions are separately introduced in \cite{Wi18},
the Lyapunov function of this paper addresses both of the strongly and non-strongly convex cases ($\sigma \geq 0$).

\subsection{Continuous-Time System}
Hereafter, the existence of the solution of the continuous-time system  \eref{SmoothingGradientFlow} for Problem~\eref{Prob}
is assumed on a considered finite or infinite time interval, which means $t\in[t_0,T]$ with $t_0 < T$ or $[t_0, +\infty)$, respectively.
\begin{remark}
	Proposition 2.1 of \cite{Qu22} provides the existence and uniqueness of
	the global solution of a similar second-order gradient-based system.
	In this proof, Proposition 6.2.1 in \cite{Ha91},
	which shows the existence and uniqueness of the global solution of
	a form $\dot{\bsx}(t) = \bsf(t,\bsx(t))$, is exploited.
	The aforementioned proposition takes into consideration
	the decreasing smoothing parameter $\mu(t)\searrow 0$ as $t\to+\infty$
	and the Lipschitz smoothness parameter of $\bsf(t,\cdot)$.
	To simplify the discussion, this paper assumes the existence of the solution.
\end{remark}

\begin{theorem}\theoremlab{GradientFlow}
	In the continuous-time system \eref{SmoothingGradientFlow} for Problem~\eref{Prob},
	the following inequality holds:
	\begin{equation}
		\F{\bsx(t)} - \F{\bsx^*}
		\leq \left(\frac{1}{2}\|\bsx(t_0)-\bsx^*\|_2^2
		+ \beta\int_{t_0}^t \ee^{\sigma(\tau-t_0)}\mu(\tau)\,\mathrm{d}\tau
		\right)
		\left(\int_{t_0}^t \ee^{\sigma(\tau-t_0)}\,\mathrm{d}\tau\right)^{-1}.
	\end{equation}
\end{theorem}
\begin{proof}
	On taking into consideration the existing Lyapunov function for the gradient flow \cite{Wi18}
	and the effect of the smooth approximation (see, e.g., \cite{Ch12b,Qu22,Bi20}),
	the following Lyapunov function candidate is introduced:
	\begin{equation}
		V(t) = \frac{\ee^{\sigma(t-t_0)}}{2}\|\bsx(t)-\bsx^*\|_2^2 +
		\int_{t_0}^t \ee^{\sigma(\tau-t_0)}\,\mathrm{d}\tau
		\left( \FF{\bsx(t)}{\mu(t)} + \beta\mu(t)- \FF{\bsx^*}{\mu(t)}
		\right).
	\end{equation}
	Its time derivative is calculated as follows:
	\begin{equation}
		\begin{split}
			\dot{V}(t) & = \frac{\sigma\ee^{\sigma(t-t_0)}}{2}\|\bsx(t)-\bsx^*\|_2^2
			+ \ee^{\sigma(t-t_0)}
			\bigl(\dot{\bsx}^\top(t)(\bsx(t)-\bsx^*)
			+ \FF{\bsx(t)}{\mu(t)} - \FF{\bsx^*}{\mu(t)}
			\bigr)
			+ \beta\ee^{\sigma(t-t_0)}\mu(t)                                         \\
			           & + \int_{t_0}^t \ee^{\sigma(\tau-t_0)}\,\mathrm{d}\tau
			\left[\nabla_{\bsx}^\top\FF{\bsx(t)}{\mu(t)}\dot{\bsx}(t) +
				\nabla_{\mu}\hh{\bsx(t)}{\mu(t)}\dot{\mu}(t) +\beta\dot{\mu}(t)
				- \nabla_{\mu}\hh{\bsx^*}{\mu(t)}\dot{\mu}(t)
				\right],
		\end{split}
		\elab{GradientFlowProof2}
	\end{equation}
	where $\nabla_{\bsx}^\top\FF{\bsx(t)}{\mu(t)} = [\nabla_{\bsx}\FF{\bsx(t)}{\mu(t)}]^\top$.
	The definition of the continuous-time system $\dot{\bsx}(t) = - \FFxgrad{\bsx(t)}{\mu(t)}$
	and the $\sigma$-strong convexity of $\FF{\cdot}{\mu(t)} = \FFO{\cdot}{\mu(t)}$
	provides the following upper bound of the second term in the right side of \eref{GradientFlowProof2}:
	\begin{equation}
		\begin{split}
			 & \dot{\bsx}^\top(t)(\bsx(t)-\bsx^*)
			+ \FF{\bsx(t)}{\mu(t)} - \FF{\bsx^*}{\mu(t)} \\
			 & =
			\left(\FF{\bsx(t)}{\mu(t)} 	+ \nabla_{\bsx}^\top\FF{\bsx(t)}{\mu(t)}(\bsx^*-\bsx(t))\right)
			-\FF{\bsx^*}{\mu(t)} \leq -\frac{\sigma}{2}\|\bsx^*-\bsx(t)\|_2^2.
		\end{split}
		\elab{GradientFlowProof3}
	\end{equation}
	The non-increase smoothing parameter sequence $\dot{\mu}(t) \leq 0$ and
	inequalities $\nabla_{\mu}\hh{\bsx(t)}{\mu(t)}\dot{\mu}(t) \leq -\beta\dot{\mu}(t)$
	and $- \nabla_{\mu}\hh{\bsx^*}{\mu(t)}\dot{\mu}(t) \leq 0$, obtained using Eq.~\eref{Definition2},
	the fourth term of the right side in Eq.~\eref{GradientFlowProof2} satisfies the following inequality:
	\begin{equation}
		\begin{split}
			\nabla_{\bsx}^\top\FF{\bsx(t)}{\mu(t)}\dot{\bsx}(t) +
			\nabla_{\mu}\hh{\bsx(t)}{\mu(t)}\dot{\mu}(t) +\beta\dot{\mu}(t)
			- \nabla_{\mu}\hh{\bsx^*}{\mu(t)}\dot{\mu}(t)
			\leq - \|\FFxgrad{\bsx(t)}{\mu(t)}\|_2^2.
		\end{split}
		\elab{GradientFlowProof4}
	\end{equation}
	The combination of Eqs.~\eref{GradientFlowProof2},\eref{GradientFlowProof3}, and \eref{GradientFlowProof4} results in
	an inequality
	\begin{equation}
		\begin{split}
			\dot{V}(t) & \leq \beta\ee^{\sigma(t-t_0)}\mu(t) -
			\int_{t_0}^t \ee^{\sigma(\tau-t_0)}\,\mathrm{d}\tau\|\FFxgrad{\bsx(t)}{\mu(t)}\|_2^2
			\leq \beta\ee^{\sigma(t-t_0)}\mu(t).
		\end{split}
	\end{equation}
	The following inequality is obtained by integrating both the side of the inequality above from $t=t_0$ to $t=T$
	and replacing $T=t$:
	\begin{equation}
		V(t) - V(t_0) \leq \beta\int_{t_0}^t \ee^{\sigma(\tau-t_0)}\mu(\tau)\,\mathrm{d}\tau.
		\elab{VtVt0}
	\end{equation}
	In $V(t)$ in the left side, the rearrangement of the item 1) of \definitionref{1}
	results in
	\begin{equation}
		\begin{split}
			\FF{\bsx(t)}{\mu(t)} + \beta\mu(t) - \FF{\bsx^*}{\mu(t)}
			\geq \F{\bsx(t)} - \F{\bsx^*}
		\end{split}
	\end{equation}
	and is applied to $V(t)$ in Eq.~\eref{VtVt0} as follows:
	\begin{equation}
		\begin{split}
			\frac{\ee^{\sigma(t-t_0)}}{2}\|\bsx(t)-\bsx^*\|_2^2 +
			\int_{t_0}^t \ee^{\sigma(\tau-t_0)}\,\mathrm{d}\tau
			\left( \F{\bsx(t)} - \F{\bsx^*}\right)
			\leq V(t_0) + \beta\int_{t_0}^t \ee^{\sigma(\tau-t_0)}\mu(\tau)\,\mathrm{d}\tau.
		\end{split}
		\elab{GradientFlowProof1}
	\end{equation}
	Consequently, the inequality of the theorem is obtained.
\end{proof}

In \theoremref{GradientFlow},
the explicit formula of the convergence rate is considered.
If an integral $\int_{t_0}^t \ee^{\sigma(\tau-t_0)}\mu(\tau)\,\mathrm{d}\tau < + \infty$ is bounded
(e.g., $\mu(t) = \ee^{-\gamma(t-t_0)}$ with $\gamma > \sigma \geq 0$),
the convergence of the upper bound is
$\mathcal{O}\big(\big(\int_{t_0}^t \ee^{\sigma(\tau-t_0)}\,\mathrm{d}\tau\big)^{-1}\big)$.
That is, the effect of the smoothing approximation does not explicitly appear
in terms of the convergence rate,
which is the same as that of the Lipschitz smooth objective function
(i.e., the smoothing approximation is not applied and results in $\mu(t) = 0\,(t\geq 0)$).
In the case of a diverging integral
$\int_{t_0}^t \ee^{\sigma(\tau-t_0)}\mu(\tau)\,\mathrm{d}\tau$,
the L'Hospital's rule and $\lim_{t\to+\infty}\mu(t) = 0$ results in
\begin{equation}
	\lim_{t\to+\infty}
	\left(\frac{1}{2}\|\bsx(t_0)-\bsx^*\|_2^2
	+\beta\int_{t_0}^t \ee^{\sigma(\tau-t_0)}\mu(\tau)\,\mathrm{d}\tau
	\right)
	\left(\int_{t_0}^t \ee^{\sigma(\tau-t_0)}\,\mathrm{d}\tau\right)^{-1}
	= \lim_{t\to+\infty} \beta\mu(t) = 0,
\end{equation}
and the upper bound converges. On summing up the aforementioned discussion,
the following corollary is obtained.
\begin{corollary}\corollarylab{GradientFlowRate}
	In the continuous-time system \eref{SmoothingGradientFlow} for Problem~\eref{Prob},
	$\lim_{t\to+\infty}\F{\bsx(t)} = \F{\bsx^*}$
	if $\mu(t)\searrow0$.
	Especially, if $\int_{t_0}^{+\infty} \ee^{\sigma(\tau-t_0)}\mu(\tau)\,\mathrm{d}\tau < + \infty$,
	\begin{equation}
		\F{\bsx(t)} - \F{\bsx^*}
		=
		\begin{cases}
			\mathcal{O}\left(\ee^{-\sigma t}\right) & (\sigma > 0), \\
			\mathcal{O}\left(1/t\right)             & (\sigma = 0).
		\end{cases}
	\end{equation}
\end{corollary}
The following proposition immediately follows by \corollaryref{GradientFlowRate} and
the $\sigma$-strong convexity of $f$,
resulting in an inequality $f(\bsx(t)) - f(\bsx^*) \geq (\sigma/2)\|\bsx(t)-\bsx^*\|_2^2$.
Beside the following proposition claims the convergence to the optimal solution,
the convergence rate of $\|\bsx(t)-\bsx^*\|_2^2$ is also obtained by
rearranging Eq.~\eref{GradientFlowProof1} (e.g., $\|\bsx(t)-\bsx^*\|_2^2 = \mathcal{O}(\ee^{-\sigma(t-t_0)})$
if $\int_{t_0}^t \ee^{\sigma(\tau-t_0)}\mu(\tau)\,\mathrm{d}\tau < + \infty$.).
\begin{proposition}\theoremlab{GradientFlowConvergence}
	In the continuous-time system \eref{SmoothingGradientFlow} for Problem~\eref{Prob},
	$\lim_{t\to+\infty} \bsx(t) =  \bsx^*$ if $f$ is strong convex (i.e., $\sigma > 0$).
\end{proposition}
\begin{remark}\remarklab{attention}
	On taking an example of the smoothing accelerated gradient flow
	and its convergence rate, given by
	$f(\bsx(t))-f(\bsx^*) = \mathcal{O}(1/t^2)$ in \cite{Qu22},
	a primitive (and somewhat misleading) substitution of the discrete step $k$ into $t$
	results in $f(\bsxI{k})-f(\bsx^*) = \mathcal{O}(1/k^2)$; this expression itself has no meaning and
	cannot be directly compared with that of the discrete-time system
	$f(\bsxI{k})-f(\bsx^*) = \mathcal{O}(\log k /k)$ in \cite{Bi20a}.
	Apart from the practical availability of the analysis solution in the continuous-time systems,
	the transform of the timescale results in an arbitrary convergence rate \cite{Us22}
	and the comparison using the indices $t$ and $k$ requires a careful discussion;
	therefore, a unified criterion providing the correspondence between
	the timeline $t$ and the discrete step $k$ is necessary.
	This study focuses on the equivalence of the discretization step in
	the forward Euler scheme and the stepsize in the discrete-time system, and
	proposes the unified correspondence of $t$ and $k$.
\end{remark}

\subsection{Discrete-Time System}
First, particular issues of the discretization step is pointed out here.
The forward Euler discretization step $s_k = t_{k+1} -t_{k}$ in the standard gradient flow,
which is equivalent to the stepsize in the standard gradient method,
for an $L$-smooth objective function $f$
is set as a constant $1/L$ (see, e.g., \cite{Wi18}).
Therefore, the iteration step $k=0,\dots,$ which is illustrated in \fref{timescaleGradient}, increase linearly with respect to
the discretized time $t_k (k=0,\dots,)$.
In contrast,
as presented in the subsequent analysis
and illustrated in \fref{timescaleSmoothingGradient},
the forward Euler discretization step of the smoothing gradient flow
(i.e., the stepsize of the smoothing gradient method)
uses a stepsize $s_k = 1/L_k = 1/(L + \alpha/\mu_k)$
so that the introduced Lyapunov function is non-increasing.
$s_k = 1/L_k$ is also non-increasing and converges to $0$ because $\mu_k$ is non-increasing and converges to $0$,
and the discretized time $t_k = \sum_{\kappa=0}^{k-1} s_{\kappa}$ in a resulting discrete-time system
obtained by the forward Euler scheme does not necessarily diverge to
$+\infty$; whether $t_k $ diverges or not depends on the design of $\mu_k$.
Besides $t_k$ does not directly appear in the discrete-time system,
the subsequent analysis shows that
the divergence of a sum $\sum_{\kappa=0}^{k-1} s_{\kappa} (=t_k)$ is
a required condition to guarantee the convergence to the optimal value.
\begin{figure}[H]
	\begin{minipage}[b]{0.48\columnwidth}
		\centering
		\begin{tikzpicture}[xscale = 1, yscale = 1]
			\draw[->,>=stealth] (0.,0)--(7.,0) node [below right] {\scriptsize$t$};
			\coordinate (t0) at (1,0);
			\draw(1,-0.1)--(1,0.1) node [below=3pt] {\scriptsize$t_0$};
			\coordinate (t1) at (2.0,0);
			\draw(2.0,-0.1)--(2.0,0.1) node [below=3pt] {\scriptsize$t_1$};
			\path[->=Stealth] (t0) edge [bend left] node [above] {\scriptsize$\frac{1}{L}$} (t1);
			\coordinate (t2) at (3.0,0);
			\draw(3.0,-0.1)--(3.0,0.1) node [below=3pt] {\scriptsize$t_2$};
			\path[->=Stealth] (t1) edge [bend left] node [above] {\scriptsize$\frac{1}{L}$} (t2);
			\coordinate (t3) at (4.0,0);
			\draw(4.0,-0.1)--(4.0,0.1) node [below=3pt] {\scriptsize$t_3$};
			\path[->=Stealth] (t2) edge [bend left] node [above] {\scriptsize$\frac{1}{L}$} (t3);
			\coordinate (t4) at (5.0,0);
			\draw(5.0,-0.1)--(5.0,0.1) node [below=3pt] {\scriptsize$t_4$};
			\path[->=Stealth] (t3) edge [bend left] node [above] {\scriptsize$\frac{1}{L}$} (t4);
			\coordinate (t5) at (6.0,0);
			\draw(6.0,-0.1)--(6.0,0.1) node [below=3pt] {\scriptsize$t_5$};
			\path[->=Stealth] (t4) edge [bend left] node [above] {\scriptsize$\frac{1}{L}$} (t5);
			\draw (6.6000000000000005,0) node [below] {\scriptsize$\cdots$};
		\end{tikzpicture}
		\caption{Discretized timeline in the standard gradient flow
			(The discretization step equivalent to the stepsize $s_k$ is a constant $1/L$.).}
		\flab{timescaleGradient}
	\end{minipage}
	\begin{minipage}[b]{0.48\columnwidth}
		\centering
		\begin{tikzpicture}[xscale = 1, yscale = 1]
			\draw[->,>=stealth] (0.,0)--(7.,0) node [below right] {\scriptsize$t$};
			\coordinate (t0) at (1,0);
			\draw(1,-0.1)--(1,0.1) node [below=3pt] {\scriptsize$t_0$};
			\coordinate (t1) at (2.0,0);
			\draw(2.0,-0.1)--(2.0,0.1) node [below=3pt] {\scriptsize$t_1$};
			\path[->=Stealth] (t0) edge [bend left] node [above] {\scriptsize$\frac{1}{L_{0}}$} (t1);
			\coordinate (t2) at (2.7071067811865475,0);
			\draw(2.7071067811865475,-0.1)--(2.7071067811865475,0.1) node [below=3pt] {\scriptsize$t_2$};
			\path[->=Stealth] (t1) edge [bend left] node [above] {\scriptsize$\frac{1}{L_{1}}$} (t2);
			\coordinate (t3) at (3.284457050376173,0);
			\draw(3.284457050376173,-0.1)--(3.284457050376173,0.1) node [below=3pt] {\scriptsize$t_3$};
			\path[->=Stealth] (t2) edge [bend left] node [above] {\scriptsize$\frac{1}{L_{2}}$} (t3);
			\coordinate (t4) at (3.784457050376173,0);
			\draw(3.784457050376173,-0.1)--(3.784457050376173,0.1) node [below=3pt] {\scriptsize$t_4$};
			\path[->=Stealth] (t3) edge [bend left] node [above] {\scriptsize$\frac{1}{L_{3}}$} (t4);
			\coordinate (t5) at (4.2316706458761315,0);
			\draw(4.2316706458761315,-0.1)--(4.2316706458761315,0.1) node [below=3pt] {\scriptsize$t_5$};
			\path[->=Stealth] (t4) edge [bend left] node [above] {\scriptsize$\frac{1}{L_{4}}$} (t5);
			\draw (4.654837710463745,0) node [below] {\scriptsize$\cdots$};
		\end{tikzpicture}
		\caption{Discretized timeline in the smoothing gradient flow
			(The discretization step equivalent to the stepsize $s_k$ is
			the reciprocal of a sequence depending on $L_k$,
			meaning the smoothing parameter sequence $\mu_k$ design.).}
		\flab{timescaleSmoothingGradient}
	\end{minipage}
\end{figure}
The subsequent analysis shows that
the convergence rate of the smoothing gradient method is given by
$\F{\bsxI{k}} - \F{\bsx^*}
	\leq
	\left.\left(\frac{1}{2}\|\bsxI{0}-\bsx^*\|_2^2  + \beta\sum_{\kappa=0}^{k-1}\mu_\kappa s_\kappa\right)
	\middle/ \sum_{\kappa=0}^{k-1}s_\kappa\right.$
for a non-strongly convex function $\sigma=0$,
while the standard gradient method for an $L$-smooth $f$
is $f(\bsxI{k}) - f(\bsx^*) \leq
	\left.
	\frac{1}{2}\|\bsxI{0}-\bsx^*\|_2^2
	\middle/
	\sum_{\kappa=0}^{k-1}s_\kappa\right.$.
Putting aside the term $\beta\sum_{\kappa=0}^{k-1}\mu_\kappa s_\kappa$, caused by the
smooth approximation, the aforementioned decreasing $s_k$
in the denominators $\sum_{\kappa=0}^{k-1}s_\kappa$ indicates that
the convergence rate of the smoothing gradient method is worse than
that of the standard gradient method.

It should be noted that the convergence analysis of the continuous-time system,
presented in the previous section, does not explicitly
include the Lipschitz smoothness parameter of
the smooth approximation $\FF{\cdot}{\mu(t)} = f(\cdot) + \widetilde{h}(\cdot,\mu(t))$;
the aforementioned issues, caused by the varying $\mu_k$, do not arise.

Subsequently, the following theorem provides a Lyapunov function of the discrete-time system.
\begin{theorem}\theoremlab{GradientFlowDiscrete}
	In the discrete-time system \eref{SmoothingGradientFlowDiscrete} for Problem~\eref{Prob} with a stepsize $s_k = 1/\LL{k}$,
	the following inequality holds at $k \geq 1$:
	\begin{equation}
		\F{\bsxI{k}} - \F{\bsx^*}
		\leq
		\left(\frac{1}{2}\|\bsxI{0}-\bsx^*\|_2^2  + \beta\sum_{\kappa=0}^{k-1}\eta_{\kappa+1} \mu_\kappa s_\kappa\right)
		\left(\sum_{\kappa=0}^{k-1}\eta_{\kappa+1} s_\kappa \right)^{-1},
		\elab{GradientFlowDiscrete}
	\end{equation}
	where $\eta_k = \prod_{l=0}^{k-1}(1-\sigma s_l)^{-1}$.
\end{theorem}
\begin{proof}
	In analogy to the Lyapunov function for the continuous-time system,
	a Lyapunov function candidate is introduced by taking into consideration the existing Lyapunov function for the standard gradient method \cite{Wi18}
	and the effect of the smooth approximation (see, e.g., \cite{Ch12b,Qu22,Bi20}) as follows:
	\begin{equation}
		\begin{split}
			V^{(k)} & = \frac{\eta_k}{2}\|\bsxI{k}-\bsx^*\|_2^2
			+
			\left(\sum_{\kappa=0}^{k-1}\eta_{\kappa+1} s_\kappa\right)
			\left(\FF{\bsxI{k}}{\mu_k} + \beta\mu_k - \FF{\bsx^*}{\mu_{k}} \right).
		\end{split}
		\elab{GradientFlowDiscreteProof5}
	\end{equation}
	For the subsequent discussion on the stepsize $s_k$ design,
	the substitution of $s_k = 1/\LL{k}$ is withheld until the last part of the proof.
	It should be noted that zero-division does not occur in
	the calculation of $\eta_k = \prod_{l=0}^{k-1}(1-\sigma s_l)^{-1}$ because
	the smoothing parameter $L_k=L+\alpha/\mu_k > L$ results in an inequality
	$1-\sigma s_k = 1-\sigma / L_k > 1 -\sigma / L \geq 0$.
	This proof shows an upper bound of the difference of the Lyapunov function candidate $V^{(k+1)}$,
	which is $V^{(k+1)} - V^{(k)}$.
	The calculation of $V^{(k+1)} - V^{(k)}$ is separated into several terms
	to simplify the calculation.
	On recalling a recursive formula $\eta_k = (1-\sigma s_k)\eta_{k+1}$ for $k \geq 0$,
	the difference of the squared $\ell_2$-norm is given by
	\begin{equation}
		\begin{split}
			\frac{\eta_{k+1}}{2}\|\bsxI{k+1}-\bsx^*\|_2^2
			- \frac{\eta_{k}}{2}\|\bsxI{k}-\bsx^*\|_2^2
			= \frac{\eta_{k+1}}{2}
			\left(
			\|\bsxI{k+1}-\bsx^*\|_2^2 - (1-\sigma s_{k})\|\bsxI{k}-\bsx^*\|_2^2
			\right).
		\end{split}
		\elab{GradientFlowDiscreteProof4}
	\end{equation}
	The following rearrangement holds:
	\begin{equation}
		\begin{split}
			 & \|\bsxI{k+1} - \bsx^* \|_2^2 - \|\bsxI{k} - \bsx^* \|_2^2
			= (\bsxI{k+1} + \bsxI{k} - 2\bsx^*)^\top(\bsxI{k+1}-\bsxI{k}) \\
			 & = (\bsxI{k+1}-\bsxI{k})^\top(\bsxI{k+1}-\bsxI{k})
			+ 2(\bsxI{k}-\bsx^*)^\top(\bsxI{k+1}-\bsxI{k})                \\
			 & = s_k^2 \|\FFxgrad{\bsxI{k}}{\mu_{k}}\|_2^2
			+2s_k\nabla_{\bsx}^\top \FF{\bsxI{k}}{\mu_{k}}(\bsx^*-\bsxI{k}).
		\end{split}
	\end{equation}
	The combination of the two equations above results in the following equation:
	\begin{equation}
		\begin{split}
			 & \frac{\eta_{k+1}}{2}\|\bsxI{k+1}-\bsx^*\|_2^2
			- \frac{\eta_{k}}{2}\|\bsxI{k}-\bsx^*\|_2^2      \\
			 & = \frac{\eta_{k+1}}{2}
			\left(
			\big[\|\bsxI{k+1}-\bsx^*\|_2^2 - \|\bsxI{k}-\bsx^*\|_2^2\big]
			+ \sigma s_{k}\|\bsxI{k}-\bsx^*\|_2^2
			\right)                                          \\
			 & = \eta_{k+1}s_k
			\left(
			\frac{s_k}{2}\|\FFxgrad{\bsxI{k}}{\mu_{k}}\|_2^2
			+\left[\nabla_{\bsx}^\top\FF{\bsxI{k}}{\mu_{k}}(\bsx^*-\bsxI{k})
				+\frac{\sigma}{2}\|\bsx^*-\bsxI{k}\|_2^2
				\right]
			\right).
		\end{split}
		\elab{GradientFlowDiscreteProof7}
	\end{equation}
	Next, the difference of the sum terms in $V^{(k+1)} - V^{(k)}$ is considered.
	On recalling $\mu_{k} \geq \mu_{k+1}$, the inequality in \definitionref{3} is used
	and results in inequalities
	$- \beta(\mu_{k}-\mu_{k+1}) \leq
		\hh{\bsx}{\mu_{k}} - \hh{\bsx}{\mu_{k+1}}
		\leq 0$ and
	$- \beta(\mu_{k}-\mu_{k+1}) \leq
		\FF{\bsx}{\mu_{k}} - \FF{\bsx}{\mu_{k+1}}
		\leq 0$, which is separated into two inequalities:
	\begin{equation}
		\begin{cases}
			\FF{\bsx}{\mu_{k+1}} + \beta\mu_{k+1} \leq \FF{\bsx}{\mu_{k}} + \beta\mu_{k}, \\
			- \FF{\bsx}{\mu_{k+1}} \leq -\FF{\bsx}{\mu_{k}}.
		\end{cases}
		\elab{GradientFlowDiscreteProof8}
	\end{equation}
	The difference related to the objective function $f$ in the Lyapunov function is
	upper bounded as the following inequality.
	On using Eq.~\eref{GradientFlowDiscreteProof8} and
	a difference
	\begin{equation}
		\sum_{\kappa=0}^{k}\eta_{\kappa+1} s_\kappa - \sum_{\kappa=0}^{k-1}\eta_{\kappa+1} s_\kappa = \eta_{k+1} s_{k},
		\elab{GradientFlowDiscreteProof9}
	\end{equation}
	the following formula holds:
	\begin{equation}
		\begin{split}
			 & \Bigl(\sum_{\kappa=0}^{k}\eta_{\kappa+1} s_\kappa\Bigr)
			\bigl(\FF{\bsxI{k+1}}{\mu_{k+1}} + \beta\mu_{k+1} - \FF{\bsx^*}{\mu_{k+1}} \bigr)
			- \Bigl(\sum_{\kappa=0}^{k-1}\eta_{\kappa+1} s_\kappa\Bigr)
			\big(\FF{\bsxI{k}}{\mu_k} + \beta\mu_k - \FF{\bsx^*}{\mu_k} \big) \\
			 & \leq \Bigl(\sum_{\kappa=0}^{k}\eta_{\kappa+1} s_\kappa\Bigr)
			\bigl(\FF{\bsxI{k+1}}{\mu_k} + \beta\mu_{k} - \FF{\bsx^*}{\mu_k} \bigr)
			- \Bigl(\sum_{\kappa=0}^{k-1}\eta_{\kappa+1} s_\kappa\Bigr)
			\bigl(\FF{\bsxI{k}}{\mu_k}+ \beta\mu_k - \FF{\bsx^*}{\mu_k}\bigr) \\
			 & = \eta_{k+1} s_{k}
			\bigl(\FF{\bsxI{k+1}}{\mu_k} + \beta\mu_{k} - \FF{\bsx^*}{\mu_k}\bigr)
			+ \Bigl(\sum_{\kappa=0}^{k-1}\eta_{\kappa+1} s_\kappa\Bigr)
			\Bigl(\FF{\bsxI{k+1}}{\mu_k}  - \FF{\bsxI{k}}{\mu_k}
			\Bigr)                                                            \\
			 & = \eta_{k+1} s_{k}
			\bigl(\FF{\bsxI{k}}{\mu_k} + \beta\mu_k - \FF{\bsx^*}{\mu_k}\bigr)+ \Bigl(\sum_{\kappa=0}^{k}\eta_{\kappa+1} s_\kappa\Bigr)
			\bigl(
			\FF{\bsxI{k+1}}{\mu_k} - \FF{\bsxI{k}}{\mu_k}
			\bigr).
		\end{split}
		\elab{GradientFlowDiscreteProof1}
	\end{equation}
	For the first term in the last line of the formula above,
	the $\sigma$-strong convexity of $\FF{\cdot}{\mu_k} = f(\cdot)+\hh{\cdot}{\mu_k}$ results in an inequality
	\begin{equation}
		\FF{\bsx^*}{\mu_k} - \FF{\bsxI{k}}{\mu_k}
		\geq \nabla_{\bsx}^\top \FF{\bsxI{k}}{\mu_k}(\bsx^*-\bsxI{k})
		+ \dfrac{\sigma}{2}\|\bsx^*-\bsxI{k}\|_2^2,
	\end{equation}
	and cancels out the difference of the squared $\ell_2$-norms in the last line of Eq.~\eref{GradientFlowDiscreteProof7}.
	In the second term of the last line of Eq.~\eref{GradientFlowDiscreteProof1},
	the $L_k$-smoothness of $\FF{\cdot}{\mu_k}=f(\cdot)+\hh{\cdot}{\mu_k}$ results in an inequality
	\begin{equation}
		\begin{split}
			\FF{\bsxI{k+1}}{\mu_{k}} - \FF{\bsxI{k}}{\mu_k}
			 & \leq \nabla_{\bsx}^\top \FF{\bsxI{k}}{\mu_k}
			(\bsxI{k+1}-\bsxI{k}) + \frac{\LL{k}}{2}\|\bsxI{k+1}-\bsxI{k}\|_2^2            \\
			 & = \left(-s_k+\frac{\LL{k}s_k^2}{2}\right)\|\FFxgrad{\bsxI{k}}{\mu_k}\|_2^2.
		\end{split}
	\end{equation}
	In the difference $V^{(k+1)} - V^{(k)}$,
	the aforementioned discussion calculates the difference of the squared $\ell_2$-norms (Eq.~\eref{GradientFlowDiscreteProof7})
	and derives an upper bound of the difference of the sums (Eq.~\eref{GradientFlowDiscreteProof1}).
	Consequently, the following upper bound is obtained:
	\begin{equation}
		\begin{split}
			 & V^{(k+1)} - V^{(k)}                                                                                                                                \\
			 & \leq \eta_{k+1}s_k
			\left(
			\frac{s_k}{2}\|\nabla_{\bsx}\FF{\bsxI{k}}{\mu_{k}}\|_2^2
			+\left[\nabla_{\bsx}^\top\FF{\bsxI{k}}{\mu_{k}}(\bsx^*-\bsxI{k})
				+\frac{\sigma}{2}\|\bsx^*-\bsxI{k}\|_2^2
				\right]
			\right)                                                                                                                                               \\
			 & \quad  +   \eta_{k+1}s_k
			\left(\beta\mu_k - \nabla_{\bsx}^\top\FF{\bsxI{k}}{\mu_{k}}(\bsx^*-\bsxI{k})
			- \dfrac{\sigma}{2}\|\bsx^*-\bsxI{k}\|_2^2 \right)                                                                                                    \\
			 & \quad + \Bigl(\sum_{\kappa=0}^{k}\eta_{\kappa+1} s_\kappa\Bigr) \bigl(-s_k+\frac{\LL{k}s_k^2}{2}\bigr)\|\nabla_{\bsx} \FF{\bsxI{k}}{\mu_{k}}\|_2^2 \\
			 & \leq \beta \eta_{k+1}\mu_k s_k
			+
			s_k\left[\frac{\eta_{k+1}s_k}{2}
				+\Bigl(\sum_{\kappa=0}^{k}\eta_{\kappa+1} s_\kappa\Bigr)\Bigl(-1+\frac{\LL{k}s_k}{2}\Bigr)
				\right]
			\|\nabla_{\bsx}\FF{\bsxI{k}}{\mu_{k}}\|_2^2                                                                                                           \\
			 & = \beta \eta_{k+1}\mu_k s_k
			+
			s_k \left[-\frac{1}{2}\Bigl(\sum_{\kappa=0}^{k-1}\eta_{\kappa+1} s_\kappa\Bigr)
				+\Bigl(\sum_{\kappa=0}^{k}\eta_{\kappa+1} s_\kappa\Bigr) \Bigl(\frac{-1+\LL{k}s_k}{2}\Bigr)
				\right]
			\|\nabla_{\bsx}\FF{\bsxI{k}}{\mu_{k}}\|_2^2,
		\end{split}
		\elab{GradientFlowDiscreteProof2}
	\end{equation}
	where the last equality follows from Eq.~\eref{GradientFlowDiscreteProof9}.
	Similar to the conventional analysis of the gradient method and related stepsize design,
	a stepsize $s_k$ satisfying $0 \leq s_k \leq 1/\LL{k}$
	makes the second term of the last line of Eq.~\eref{GradientFlowDiscreteProof2} non-positive.
	On summing up the inequality \eref{GradientFlowDiscreteProof2} for $k$, an inequality
	$V^{(k)} - V^{(0)} \leq \beta\sum_{\kappa=0}^{k-1}\eta_{\kappa+1} \mu_\kappa s_\kappa$ follows.
	The first inequality in \definitionref{1} and this inequality
	result in the first and second inequalities of the following formula, respectively:
	\begin{equation}
		\begin{split}
			 & \frac{\eta_k}{2}\|\bsxI{k}-\bsx^*\|_2^2
			+
			\Bigl(\sum_{\kappa=0}^{k-1}\eta_{\kappa+1} s_\kappa\Bigr)
			\bigl(\F{\bsxI{k}} - \F{\bsx^*}\bigr)                                                       \\
			 & \leq
			\frac{\eta_k}{2}\|\bsxI{k}-\bsx^*\|_2^2
			+
			\Bigl(\sum_{\kappa=0}^{k-1}\eta_{\kappa+1} s_\kappa\Bigr)
			\bigl(\FF{\bsxI{k}}{\mu_k} + \beta\mu_k - \FF{\bsx^*}{\mu_k}\bigr)                          \\
			 & = V^{(k)} \leq V^{(0)} + \beta\sum_{\kappa=0}^{k-1}\eta_{\kappa+1} \mu_\kappa s_\kappa =
			\frac{1}{2}\|\bsxI{0}-\bsx^*\|_2^2  + \beta\sum_{\kappa=0}^{k-1}\eta_{\kappa+1} \mu_\kappa s_\kappa.
		\end{split}
		\elab{GradientFlowDiscreteProof6}
	\end{equation}
	On rearranging the inequality above, the inequality of the theorem is obtained.
\end{proof}
The aforementioned proof indicates that the stepsize $s_k = 1/\LL{k}$, which depends on the smoothing parameter $\mu_k$, is chosen so that
the coefficient of a squared $\ell_2$-norm $\|\nabla_{\bsx}\FF{\bsxI{k}}{\mu_{k}}\|_2^2$ in Eq.~\eref{GradientFlowDiscreteProof2} is non-positive.
The upper bound $\left((1/2)\|\bsxI{0}-\bsx^*\|_2^2  + \beta\sum_{\kappa=0}^{k-1}\eta_{\kappa+1} \mu_\kappa s_\kappa\right)
	\left(\sum_{\kappa=0}^{k-1}\eta_{\kappa+1} s_\kappa \right)^{-1}$
in \theoremref{GradientFlowDiscrete} seems somewhat complicated because of
$\eta_k = \prod_{l=0}^{k-1}(1-\sigma s_l)^{-1}$, depending on $s_k = 1/L_{k} = (L+\alpha/\mu_k)^{-1}$.

A case that $\sum_{\kappa=0}^{k-1}\eta_{\kappa+1} s_\kappa$ diverges to $+\infty$ is taken into consideration here.
\begin{itemize}
	\item  If $\sum_{\kappa=0}^{k-1}\eta_{\kappa+1} \mu_\kappa s_\kappa < + \infty$, the aforementioned upper bound converges to $0$.
	\item  If $\sum_{\kappa=0}^{k-1}\eta_{\kappa+1} \mu_\kappa s_\kappa$ also diverges to $+\infty$,
	      the Stolz--Ces\`{a}ro theorem (i.e., a discrete sequence version of the L'Hospital's rule,
	      see, e.g., \cite{Fu13}) with $\mu_k\searrow0$
	      shows the convergence to the optimal value:
	      \begin{equation}
		      \lim_{k\to+\infty}
		      \left(\frac{1}{2}\|\bsxI{0}-\bsx^*\|_2^2  + \beta\sum_{\kappa=0}^{k-1}\eta_{\kappa+1} \mu_\kappa s_\kappa\right)
		      \left(\sum_{\kappa=0}^{k-1}\eta_{\kappa+1} s_\kappa \right)^{-1}
		      = \lim_{k\to+\infty}\beta\mu_{k+1} = 0.
		      \elab{limit1}
	      \end{equation}
\end{itemize}
Therefore, if $\sum_{\kappa=0}^{+\infty}\eta_{\kappa+1} s_\kappa$ diverges, the convergence to the optimal solution is achieved.
However, whether $\sum_{\kappa=0}^{+\infty}\eta_{\kappa+1} s_\kappa$ diverges or is bounded is still not clear
because of the complicated $\eta_k$.
The following lemma claims that the convergence is equivalent to the convergence of $\sum_{\kappa=0}^{+\infty}s_\kappa$.
\begin{lemma}\lemmalab{etaSum}
	For a stepsize $s_k = 1/\LL{k}$,
	$\sum_{\kappa=0}^{k-1}s_\kappa$ converges if and only if
	$\sum_{\kappa=0}^{k-1}\eta_{\kappa+1}  s_\kappa$ converges.
\end{lemma}
\begin{proof}
	The case of $\sigma = 0$, meaning a non-strong convex $f$, results in $\eta_k = 1$,
	and the two sums are the same; therefore, the case of $\sigma > 0$ is taken into consideration here.
	It should be noted that the sums converge or diverge
	because $\sum_{\kappa=0}^{k-1}s_\kappa$ and $\sum_{\kappa=0}^{k-1}\eta_{\kappa+1} s_\kappa$
	are both strictly increasing.
	This proof shows that the convergence of each sum is equivalent.
	The first-order Taylor expansion of the exponential $\ee^{-\delta} \geq 1 - \delta$
	results in $\ee^{\delta} \leq (1 - \delta)^{-1}$ for $\delta < 1$ by taking the reciprocal.
	The substitution of $\delta = \sigma s_k= \sigma / (L + \alpha/\mu_k) < \sigma/L \leq 1$
	results in the following lower bound of $\eta_k$:
	\begin{equation}
		\ee^{\sigma(s_{k-1}+\dots+s_0)}
		\leq \prod_{l=0}^{k-1}(1-\sigma s_l)^{-1} = \eta_k.
		\elab{eta}
	\end{equation}
	Therefore, $\eta_k \geq \ee^{\sigma(s_{k-1}+\cdots+s_0)} > 1$ and
	$\sum_{\kappa=0}^{k-1}s_\kappa < \sum_{\kappa=0}^{k-1}\eta_{\kappa+1} s_\kappa$.
	If $\sum_{\kappa=0}^{k-1}\eta_{\kappa+1} s_\kappa$ converges,
	$\sum_{\kappa=0}^{k-1}s_\kappa$ also converges.

	The opposite is shown in the remainder of this proof.
	$\sum_{\kappa=0}^{+\infty}\eta_{\kappa+1} s_\kappa$ is bounded
	if $\eta_k$ and $\sum_{\kappa=0}^{k-1} s_\kappa$ are bounded
	because $\sum_{\kappa=0}^{k-1}\eta_{\kappa+1} s_\kappa
		\leq \max_{\kappa=1,\dots,k}\eta_\kappa
		\left(\sum_{\kappa=0}^{k-1}s_\kappa \right)$.
	There exists a sufficiently large $N$ satisfying an inequality because $\sum_{\kappa=0}^{N-1}s_\kappa $ converges:
	\begin{equation}
		\sum_{\kappa=N}^{+\infty} s_\kappa < 1/\sigma
		\elab{etaProof1}
	\end{equation}
	The reciprocal of $\eta_k/ \eta_N = \prod_{l=N}^{k-1}(1-\sigma s_l)^{-1}$ is lower-bounded as follows:
	\begin{equation}
		\begin{split}
			\eta_N / \eta_k
			 & = \prod_{l=N}^{k-1}(1-\sigma s_l)
			= (1-\sigma s_{k-1})(1-\sigma s_{k-2})\prod_{l=N}^{k-3}(1-\sigma s_l)                          \\
			 & = (1 - \sigma(s_{k-1} + s_{k-2}) + \sigma^2 s_{k-1} s_{k-2})\prod_{l=N}^{k-3}(1-\sigma s_l) \\
			 & > (1 - \sigma(s_{k-1} + s_{k-2}))\prod_{l=N}^{k-3}(1-\sigma s_l)
			> \dots
			> 1 - \sigma\sum_{l = N}^{k-1}s_l > 0.
		\end{split}
	\end{equation}
	The inequality above results in $\eta_k < \eta_N (1 - \sigma\sum_{l = N}^{k-1}s_l)^{-1}$.
	On substituting $k = N$ in Eq.~\eref{etaProof1}, an inequality $\sum_{\kappa=N}^{+\infty} s_\kappa < 1/\sigma$ hold,
	and taking the limit $k\to+\infty$ of $\eta_k < \eta_N (1 - \sigma\sum_{l = N}^{k-1}s_l)^{-1}$
	shows $\eta_k < + \infty$. Consequently,
	the divergence of $\sum_{\kappa=0}^{k-1}s_\kappa$ and $\sum_{\kappa=0}^{k-1}\eta_{\kappa+1}s_\kappa$ is equivalent,
	and the claim of the lemma is obtained.
\end{proof}
On using Eq.~\eref{limit1} and \lemmaref{etaSum},
the following corollary of \theoremref{GradientFlowDiscrete} is obtained.
It should be noted that the corollary does not depend on $\sigma$.
\begin{corollary}\corollarylab{SmoothingGradientFlowDiscretefvalConvergence}
	In the discrete-time system \eref{SmoothingGradientFlowDiscrete} for Problem~\eref{Prob} with a stepsize $s_k = 1/\LL{k}$,
	if $\sum_{\kappa=0}^{k-1}s_\kappa$ diverges to $+\infty$,
	then $\lim_{k\to+\infty}\F{\bsxI{k}} = \F{\bsx^*}$.
\end{corollary}
Furthermore, the following proposition shows the convergence to the optimal solution.
\begin{proposition}\propositionlab{GradientFlowDiscreteConvergence}
	In the discrete-time system \eref{SmoothingGradientFlowDiscrete} for Problem~\eref{Prob} with $\sigma > 0$ and a stepsize $s_k = 1/\LL{k}$,
	if $\sum_{\kappa=0}^{k-1}s_\kappa$ diverges to $+\infty$,
	then $\lim_{k\to+\infty} \bsxI{k} = \bsx^*$.
\end{proposition}
\begin{proof}
	Eq.~\eref{GradientFlowDiscreteProof6} and $\F{\bsxI{k}} - \F{\bsx^*} \geq 0$ result in
	\begin{equation}
		\frac{\eta_k}{2}\|\bsxI{k}-\bsx^*\|_2^2
		\leq \frac{1}{2}\|\bsxI{0}-\bsx^*\|_2^2  + \beta\sum_{\kappa=0}^{k-1}\eta_{\kappa+1} \mu_\kappa s_\kappa,
	\end{equation}
	and $(1/2)\|\bsxI{k}-\bsx^*\|_2^2$ has the following upper bound:
	\begin{equation}
		\frac{1}{2}\|\bsxI{k}-\bsx^*\|_2^2 \leq
		\frac{1}{2\eta_k}\|\bsxI{0}-\bsx^*\|_2^2  +
		\frac{\beta}{\eta_k}
		\sum_{\kappa=0}^{k-1}\eta_{\kappa+1} \mu_\kappa s_\kappa.
		\elab{GradientFlowDiscreteConvergenceProof1}
	\end{equation}
	The limit $k\to+\infty$ of the right side is taken into consideration.
	Eq.~\eref{eta} indicates $\eta_k \geq \ee^{\sigma(s_{k-1}+\cdots+s_0)}$,
	and the first term converges to $0$.
	For the second term, the Stolz-Ces\`{a}ro theorem shows the following limit of $\mu_k\searrow0$:
	\begin{equation}
		\lim_{k\to+\infty}
		\frac{\beta}{\eta_k}
		\sum_{\kappa=0}^{k-1}\eta_{\kappa+1} \mu_\kappa s_\kappa
		= \lim_{k\to+\infty}\frac{\beta\eta_{k+1}\mu_{k}s_{k}}{\eta_{k+1} - \eta_{k}}
		= \lim_{k\to+\infty}\frac{\beta\eta_{k+1}\mu_{k}s_{k}}{\eta_{k+1}(1 - [1-\sigma s_{k}])}
		= \lim_{k\to+\infty}\frac{\beta\mu_{k}}{\sigma} = 0.
	\end{equation}
	Therefore, $\bsxI{k}$ converges to the optimal solution $\bsx^*$.
\end{proof}
The following corollary summarizes the convergence rate of the discrete-time system.
\begin{corollary}\corollarylab{GradientFlowDiscreteNonStrongly}
	In the discrete-time system \eref{SmoothingGradientFlowDiscrete} for Problem~\eref{Prob} with a stepsize $s_k = 1/\LL{k}$, and a smoothing parameter sequence
	$\mu_k = \mu_0 (k+1)^{\gamma}\,(0 < \gamma \leq 1, \, k \geq 0)$,
	the following inequality holds for $k \geq 1$:
	\begin{equation}
		\begin{split}
			 & \F{\bsxI{k}} - \F{\bsx^*} \\
			 & \leq
			\begin{cases}
				\dfrac{(1/2)\|\bsx^*-\bsxI{0}\|_2^2 + \alpha^{-1}\beta\mu_0^2(1-2\gamma)^{-1}(k^{1-2\gamma}-2\gamma)}
				{(L + \alpha\mu_0^{-1})^{-1}(1-\gamma)^{-1}\left[(k+1)^{1-\gamma}-1\right]}
				= \mathcal{O}(k^{-\gamma})
				 & \left(0 < \gamma < \dfrac{1}{2}, \dfrac{1}{2} < \gamma <1\right), \\
				\dfrac{(1/2)\|\bsx^*-\bsxI{0}\|_2^2 + \alpha^{-1}\beta\mu_0^2(1+\log k)}
				{2(L + \alpha\mu_0^{-1})^{-1}\left[(k+1)^{1/2}-1\right]}
				= \mathcal{O}(k^{-1/2}\log k)
				 & \left(\gamma = \dfrac{1}{2}\right),                               \\
				\dfrac{(1/2)\|\bsx^*-\bsxI{0}\|_2^2
				+ \alpha^{-1}\beta\mu_0^2(2-k^{-1})}
				{(L + \alpha\mu_0^{-1})^{-1}\log(k+1)}
				= \mathcal{O}(1/\log k)
				 & \left(\gamma = 1\right).
			\end{cases}
		\end{split}
	\end{equation}
\end{corollary}
\begin{proof}
	On recalling $\sigma=0$ and $\eta_k=1$,
	the upper bound of $\F{\bsx(t)} - \F{\bsx^*}$
	in Eq.~\eref{GradientFlowDiscrete} of \theoremref{GradientFlowDiscrete}
	is evaluated.
	The stepsize $s_\kappa = 1/L_k = 1/(L+\alpha\mu_0(\kappa+1)^{-1})$
	satisfies the following inequality:
	\begin{equation}
		\frac{1}{(L + \alpha\mu_0^{-1})(\kappa+1)^\gamma}
		\leq
		\frac{1}{L + \alpha\mu_0^{-1}(\kappa+1)^\gamma}
		\leq
		\frac{1}{\alpha\mu_0^{-1}(\kappa+1)^\gamma}.
	\end{equation}
	The sum of the first inequality above is lower-bounded as follows:
	\begin{equation}
		\begin{split}
			 & \sum_{\kappa=0}^{k-1} s_\kappa = \sum_{\kappa=0}^{k-1} \frac{1}{L_\kappa}
			\geq \sum_{\kappa=0}^{k-1} \frac{1}{(L+\alpha\mu_0^{-1})(\kappa+1)^\gamma}
			\geq \int_{0}^{k} \frac{1}{(L + \alpha\mu_0^{-1})(\kappa+1)^\gamma}\,\mathrm{d}\kappa \\
			 & =
			\begin{cases}{}
				\left[
					\dfrac{(1-\gamma)^{-1}(\kappa+1)^{1-\gamma}}{L + \alpha\mu_0^{-1}}
					\right]_0^{k}
				= \dfrac{(1-\gamma)^{-1}\left[(k+1)^{1-\gamma}-1\right]}{L + \alpha\mu_0^{-1}} &
				(\gamma \neq 1),                                                                               \\
				\left[
					\dfrac{\log (\kappa+1)}{L + \alpha\mu_0^{-1}}
					\right]_0^{k}
				= \dfrac{\log (k+1)}{L + \alpha\mu_0^{-1}}                                     & (\gamma = 1).
			\end{cases}
		\end{split}
		\elab{GradientFlowDiscreteBoundNonStronglyProof1}
	\end{equation}
	Similarly, a term
	$\beta\eta_{k+1} \mu_k s_k = \beta\mu_{\kappa}/\LL{\kappa}
		= \beta \mu_0/(L (\kappa+1)^\gamma + \alpha\mu_0^{-1}(\kappa+1)^{2\gamma})$ in
	the numerator of Eq.~\eref{GradientFlowDiscrete} is bounded as follows:
	\begin{equation}
		\frac{\beta \mu_0}{(L + \alpha\mu_0^{-1})(\kappa+1)^{2\gamma}}
		\leq
		\frac{\beta \mu_0}{L (\kappa+1)^\gamma + \alpha\mu_0^{-1}(\kappa+1)^{2\gamma}}
		\leq
		\frac{\beta \mu_0^2}{\alpha (\kappa+1)^{2\gamma}}.
	\end{equation}
	The sum is upper-bounded as follows:
	\begin{equation}
		\begin{split}
			 & \beta\sum_{\kappa=0}^{k-1}\frac{\mu_\kappa}{\LL{\kappa}}
			\leq \beta \sum_{\kappa=0}^{k-1} \frac{\mu_0^2}{\alpha (\kappa+1)^{2\gamma}}
			\leq \frac{\beta\mu_0^2}{\alpha}
			+ \int_{0}^{k-1}
			\frac{\beta\mu_0^2}{\alpha (\kappa +1)^{2\gamma}}
			\,\mathrm{d}\kappa                                          \\
			 & =
			\begin{cases}{}
				\dfrac{\beta\mu_0^2}{\alpha}
				\left( 1 + \Big[ (1-2\gamma)^{-1} (\kappa+1)^{1-2\gamma} \Big]_0^{k-1}
				\right)
				= \alpha^{-1}\beta\mu_0^2(1-2\gamma)^{-1}(k^{1-2\gamma}-2\gamma) & (\gamma \neq 1/2), \\
				\dfrac{\beta\mu_0^2}{\alpha} \left( 1 + \Big[ \log (\kappa+1) \Big]_0^{k-1}
				\right)
				= \alpha^{-1}\beta\mu_0^2(1+\log k)                              & (\gamma = 1/2).
			\end{cases}
		\end{split}
		\elab{GradientFlowDiscreteBoundNonStronglyProof2}
	\end{equation}
	Consequently,
	the upper bound of $\beta\sum_{\kappa=0}^{k-1}\mu_\kappa/\LL{\kappa}$ above
	and lower bound of $\sum_{\kappa=0}^{k-1} 1/\LL{\kappa}$, given in Eq.~\eref{GradientFlowDiscreteBoundNonStronglyProof1}, are applied to
	the upper bound of $\F{\bsx(t)} - \F{\bsx^*}$ in Eq.~\eref{GradientFlowDiscrete} and result in the claim of the corollary.
\end{proof}

\section{Comparison of Continuous- and Discrete-Time Smoothing Gradient Flow}
\subsection{Relationship between Timeline and Stepsize}
As discussed in the previous sections,
the stepsize $s_k=1/L_k\,(k=0,1,\dots,)$ means that
the continuous timeline is discretized as $t_{k+1} = t_{k} + s_k$ and
the continuous-time system is numerically integrated using the forward Euler scheme,
which is equivalent to the smoothing gradient method.
The discretized time point $t_k$, which corresponds to the discrete $k$-th step of the smoothing gradient method,
satisfies a difference equation
\begin{equation}
	s_k = t_{k+1} - t_k = \frac{1}{L + \alpha/\mu(t_k)},
	\elab{sk}
\end{equation}
and its solution explicitly describes the relationship between
the continuous $t$ and the discrete $k$.
However, an explicit expression of the difference equation above
is difficult to obtain.
Instead, the remainder of this section derives upper and lower bounds of $t_k, k$, and $\mu(t_k)$.

In the standard gradient method (i.e., a constant stepsize $s_k = 1/L$ for an $L$-smooth objective function)
or a fixed smoothing parameter $s_k = 1/(L+\alpha/\mu)$, persuading an $\varepsilon$-optimal solution (Section 10.8 of \cite{Be17a}),
the discretized time point $t_k$ grows linearly with respect to the discrete $k$.
Therefore, the aforementioned relationship of $t_k$ and $k$ is a particular issue in discretizing
the smoothing gradient flow the Lipschitz smoothness parameter $L_k$ of which varies at each $k$-th step.

The following theorem shows the case of an exponentially decaying smoothing parameter.
\begin{theorem}\theoremlab{case1}
	In the stepsize rule \eref{sk},
	a smoothing parameter sequence $\mu(t_k) = \mu_k = \mu_0\lambda^{k}\,(\lambda\in(0,1),\,\mu_0 > 0)$
	satisfies the following inequality for $k \geq 1$:
	\begin{equation}
		\begin{cases}
			(L+\alpha\mu_0^{-1})^{-1}\dfrac{1-\lambda^{k}}{1-\lambda}
			\leq
			t_{k} - t_{0}
			\leq
			(\alpha\mu_0^{-1})^{-1}\dfrac{1-\lambda^{k}}{1-\lambda} < +\infty, \\
			\mu_0 - (\mu_0 L + \alpha)(1-\lambda)(t_k-t_0)
			\leq \mu(t_k)
			\leq  \mu_0 - \alpha(1-\lambda)(t_k-t_0).
		\end{cases}
	\end{equation}
\end{theorem}
\begin{proof}
	Upper and lower bounds of the stepsize are given by
	\begin{equation}
		(L+\alpha\mu_0^{-1})^{-1}\lambda^{k}
		\leq
		t_{k+1} - t_{k} = \frac{1}{L + \alpha\mu_0^{-1}\lambda^{-k}}
		\leq
		(\alpha\mu_0^{-1})^{-1}\lambda^{k}.
	\end{equation}
	Summing up the inequality above for $k \geq 1$ results in
	\begin{equation}
		(L+\alpha\mu_0^{-1})^{-1}\frac{1-\lambda^{k}}{1-\lambda}
		\leq
		t_{k} - t_{0}
		\leq
		(\alpha\mu_0^{-1})^{-1}\frac{1-\lambda^{k}}{1-\lambda},
	\end{equation}
	and the first inequality of the theorem is obtained. On rearranging the inequality above,
	the second inequality of the theorem is also obtained.
\end{proof}
\begin{example}\examplelab{1}
	As a simple example, the minimization of the absolute value of a scalar variable
	(i.e., a problem $\min_{x\in\R}|x|$) is addressed.
	A smooth approximation is a $(1/\mu)$-smooth approximation of $|x|$
	with parameters $(\alpha,\beta) = (1,1)$ given by $f(x,\mu) = \sqrt{x^2+\mu^2}-\mu$.
	The partial derivative of $f(x,\mu)$ with respect to a smoothing parameter $\mu > 0$ is
	\begin{equation}
		\nabla_\mu f(x,\mu) = \frac{\mu}{\sqrt{x^2 + \mu^2}} -1
		\in[-1, 0),
	\end{equation}
	and $f(x,\mu)$ is a $(1/\mu)$-Lipschitz continuous smooth approximation with
	parameters $(1,1)$. The updating formula of the discrete-time system is
	\begin{equation*}
		x^{(k+1)} = x^{(k)} - s_k\frac{x^{(k)}}{\sqrt{(x^{(k)})^2+\mu_{k}^2}},
	\end{equation*}
	and the Lipschitz smoothness parameter of $|x|$ is $L=0$ and results in
	$s_k = 1/(L+\alpha\mu_k^{-1}) = \mu_k$. Therefore,
	\begin{equation*}
		0 \leq \left|\frac{x^{(k)}}{\sqrt{(x^{(k)})^2+\mu_0^2 \lambda^{2k}}}\right| \leq 1,
	\end{equation*}
	and an inequality $|x^{(k+1)} - x^{(k)}| \leq s_k$ is used to upper-bound $|x^{(0)} - x^{(+\infty)}|$ as follows:
	\begin{equation*}
		|x^{(0)} - x^{(+\infty)}|
		\leq \sum_{k=0}^{+\infty}s_k = \sum_{k=0}^{+\infty} \mu_0 \lambda^{k}
		= \frac{\mu_0}{1-\lambda}< +\infty.
	\end{equation*}
	That is, the domain that $x^{(k)}$ can reach is restricted;
	therefore, an initial point $x^{(0)}$ far from the optimal solution $x^* = 0$ (more precisely, a case of $|x^{(0)}| > \mu_0/(1-\lambda)$)
	cannot reach the solution. In addition, Theorem \ref{theorem:GradientFlowDiscrete} and \lemmaref{etaSum} cannot establish the convergence
	for an upper-bounded sum $\sum_{k=0}^{+\infty}s_k < +\infty$;
	the aforementioned result is consistent with the claim of \theoremref{GradientFlowDiscrete}.

	Conversely, the expression of the smoothing gradient flow (the continuous-time system) is derived from the smoothing gradient method (the discrete-time system).
	That is, the smoothing parameter function $\mu(t)$
	that generates the smoothing parameter sequence $\mu(t_k) = \mu_k = \mu_0\lambda^k$
	in the discrete-time system with the forward Euler scheme is explored.
	The application of \theoremref{case1} with $L=0$ and $\alpha=1$ results in $\mu(t_k) = \mu_0 - (1-\lambda)(t_k-t_0)$;
	therefore, the continuous-time smoothing parameter $\mu(t) = \mu_0-(1-\lambda)(t-t_0)$
	in the continuous-time system and
	resulting discrete-time system obtained using the forward Euler scheme is equivalent to the smoothing gradient method.
	In the continuous timeline, at a time period $t_{+\infty} = \mu_0/(1-\lambda)+t_0 = \lim_{k\to+\infty} t_k$ (i.e., the last time period that the discretized system can reach),
	$\mu(t_{+\infty}) = 0$ means
	that the ODE of the continuous-time system becomes ill-posed at $t=t_{+\infty}$.
\end{example}

The convergence of another form $\mu_k = \mu_0(k+1)^{-\gamma}$ is summarized in the following proposition.
\begin{proposition}\propositionlab{case2}
	In the stepsize rule \eref{sk},
	a smoothing parameter sequence $\mu(t_k) = \mu_k = \mu_0(k+1)^{-\gamma}$
	satisfies the following inequality for $k \geq 1$:
	\begin{enumerate}
		\item For $\gamma\neq1$,
		      \begin{equation}
			      \begin{cases}
				      (L+\alpha\mu_0^{-1})^{-1}(1-\gamma)^{-1}[(k+1)^{1-\gamma}-1]
				      \leq t_k - t_0
				      \leq (\alpha\mu_0^{-1})^{-1}(1-\gamma)^{-1}[k^{1-\gamma}-\gamma],       \\
				      [\alpha\mu_0^{-1}(1-\gamma)(t_k-t_0)+\gamma]^{(1-\gamma)^{-1}}
				      \leq k
				      \leq [(L+\alpha\mu_0^{-1})(1-\gamma)(t_k-t_0)+1]^{(1-\gamma)^{-1}} - 1, \\
				      \mu_0[(L+\alpha\mu_0^{-1})(1-\gamma)(t_k-t_0)+1]^{-\frac{\gamma}{1-\gamma}}
				      \leq
				      \mu(t_k)
				      \leq \mu_0\left([\alpha\mu_0^{-1}(1-\gamma)(t_k-t_0)+\gamma]^{(1-\gamma)^{-1}}+1\right)^{-\gamma}.
			      \end{cases}
		      \end{equation}
		\item For $\gamma = 1$,
		      \begin{equation}
			      \begin{cases}
				      (L+\alpha\mu_0^{-1})^{-1}\log(k+1)
				      \leq t_k - t_0
				      \leq (\alpha\mu_0^{-1})^{-1}(1 + \log k),      \\
				      \exp([\alpha\mu_0^{-1}][t_k -t_0]-1)
				      \leq k
				      \leq \exp([L+\alpha\mu_0^{-1}][t_k -t_0]) - 1, \\
				      \mu_0\exp(-[L+\alpha\mu_0^{-1}][t_k -t_0])
				      \leq \mu(t_k)
				      \leq
				      \mu_0\left[\exp([\alpha\mu_0^{-1}][t_k -t_0]-1)+1\right]^{-1}.
			      \end{cases}
		      \end{equation}
	\end{enumerate}
\end{proposition}
\begin{proof}
	Similar to the proof of \theoremref{case1}, upper and lower bounds
	\begin{equation}
		(L+\alpha\mu_0^{-1})^{-1}(k+1)^{-\gamma}
		\leq
		t_{k+1} - t_{k} = \frac{1}{L + \alpha\mu_0^{-1}(k+1)^\gamma}
		\leq
		(\alpha\mu_0^{-1})^{-1}(k+1)^{-\gamma}
	\end{equation}
	are used and result in bounds of the sum as follows:
	\begin{equation}
		(L+\alpha\mu_0^{-1})^{-1}\sum_{s=0}^{k-1} (s+1)^{-\gamma}
		\leq
		t_{k} - t_{0}
		\leq
		(\alpha\mu_0^{-1})^{-1}\sum_{s=0}^{k-1}(s+1)^{-\gamma}.
		\elab{case2Proof1}
	\end{equation}
	For $k \geq 1$,
	\begin{align}
		 & \int_0^{k} (u+1)^{-\gamma}\,\mathrm{d}u
		=
		\begin{cases}
			(1-\gamma)^{-1}[(k+1)^{1-\gamma}-1] & (\gamma\neq 1), \\
			\log(k+1)                           & (\gamma = 1)
		\end{cases} \\
		 & \leq
		\sum_{s=0}^{k-1} (s+1)^{-\gamma}
		\leq 1+\int_{0}^{k-1} (u+1)^{-\gamma}\,\mathrm{d}u
		= \begin{cases}
			  (1-\gamma)^{-1}[k^{1-\gamma}-\gamma] & (\gamma\neq 1), \\
			  1 + \log k                           & (\gamma = 1),
		  \end{cases}
		\elab{case2Proof2}
	\end{align}
	and Eqs.~\eref{case2Proof1} and \eref{case2Proof2} result in
	and the first formulas of both the items 1) and 2).
	The second and third formulae of each item is obtained by rearranging the first formula.
\end{proof}

\begin{remark}\remarklab{5}
	Besides the formulae for $\gamma \neq 1$ in the theorem hold for $\gamma < 1$,
	the following formulae provided $\gamma -1 >0$ is given here
	for simplifying subsequent discussion:
	\begin{equation}
		\begin{cases}
			(L+\alpha\mu_0^{-1})^{-1}(\gamma-1)^{-1}[1-(k+1)^{-(\gamma-1)}]
			\leq t_k - t_0
			\leq (\alpha\mu_0^{-1})^{-1}(\gamma-1)^{-1}[\gamma - k^{-(\gamma-1)}] < +\infty, \\
			[\gamma - \alpha\mu_0^{-1}(\gamma-1)(t_k-t_0)]^{-(\gamma-1)^{-1}}
			\leq k
			\leq [1-(L+\alpha\mu_0^{-1})(\gamma-1)(t_k-t_0)]^{-(\gamma-1)^{-1}} - 1,         \\
			\mu_0[1-(L+\alpha\mu_0^{-1})(\gamma-1)(t_k-t_0)]^{\frac{\gamma}{\gamma-1}}
			\leq
			\mu(t_k)
			\leq \mu_0\left([\gamma - \alpha\mu_0^{-1}(\gamma-1)(t_k-t_0)]^{-(\gamma-1)^{-1}}+1\right)^{-\gamma}.
		\end{cases}
		\elab{remark5}
	\end{equation}
\end{remark}

\begin{example}\examplelab{2}
	The same example of \exampleref{1} is taken into consideration.
	For $\mu_k = \mu_0 (k+1)^{-\gamma}\, (\mu_0 > 0,\,\gamma > 1)$
	with a large value of $x^{(0)}$,
	the state variable and objective function value cannot reach zero.
	Examples \ref{example:1} and \ref{example:2} do not satisfy
	the condition: $\sum_{\kappa=0}^{k-1}s_\kappa$ diverges as $k\to+\infty$ in \theoremref{GradientFlowDiscrete}.
	In the continuous-time system,
	the upper and lower bounds of $\mu(t_k)$ given by Eq.~\eref{remark5} in \remarkref{5}
	indicate that $\mu(t)$ reaches $\mu(t)=0$ at finite $t$.
\end{example}

\subsection{Comparison on Convergence Rate}
This section takes into consideration the case of $\sigma=0$ for simplicity
and compares the convergence rate of the objective function value
in the continuous- and discrete-time systems.

\corollaryref{GradientFlowRate} with $\sigma=0$ shows the convergence rate
of the objective function value in the continuous-time system $\mathcal{O}(1/t)$.
As discussed in \remarkref{attention}, the convergence rate in the discrete-time system
$\mathcal{O}(k^{-1/2}\log k)$ in \corollaryref{GradientFlowDiscreteNonStrongly}
cannot be (somehow mistakenly) compared with $\mathcal{O}(1/k)$, which is obtained by
replacing $t$ in $\mathcal{O}(1/t)$ with $k$.
Therefore, an ideal situation that the analytic solution of the continuous-time system is somehow available
is assumed.
The sample values of the analytic solution $\bsx(t_k)$ at the continuous timeline $t_k (k=0,1,\dots)$
and $\bsxI{k}$ of the discrete-time system with a smoothing parameter $\mu_k = \mu(k+1)^{-\gamma}\,(\gamma \in (0,1])$
are compared in terms of the convergence rate, as illustrated in \fref{sample}.
\propositionref{case2} shows that
\begin{equation}
	\left\{
	\begin{array}{l@{\,\,}l@{\quad}l}
		\mu(t) = \Theta\big((t-t_0)^{-\gamma/(1-\gamma)}\big),  & t_k - t_0 = \Theta\big((k+1)^{1-\gamma}\big) & (0 < \gamma < 1), \\
		\mu(t) = \mathcal{O}\big(\exp(-\alpha\mu_0^{-1}t)\big), & t_k - t_0 = \Theta\big(\log(k+1)\big)        & (\gamma = 1).
	\end{array}
	\right.
\end{equation}
On recalling $\gamma/(1-\gamma) = 1$ if $\gamma = 1/2$,
the following upper bound is obtained:
\begin{equation}
	\begin{split}
		 & \F{\bsx(t_k)} - \F{\bsx^*}
		\leq \left(\frac{1}{2}\|\bsx(t_0)-\bsx^*\|_2^2
		+ \beta\int_{t_0}^{t_k} \mu(\tau)\,\mathrm{d}\tau
		\right)
		\left(t_k- t_0\right)^{-1}    \\
		 & =
		\begin{cases}
			\Theta((t_k-t_0)^{-\gamma/(1-\gamma)}) = \Theta((k+1)^{-\gamma})      & (0 < \gamma < 1/2,\, 1/2 < \gamma < 1), \\
			\Theta((t_k-t_0)^{-1}\log(t_k - t_0)) = \Theta((k+1)^{-1/2}\log(k+1)) & (\gamma = 1/2),                         \\
			\Theta((t_k-t_0)) = \Theta(1/\log(k+1))                               & (\gamma = 1).
		\end{cases}
	\end{split}
\end{equation}
All the cases above are consistent with \corollaryref{GradientFlowDiscreteNonStrongly}.
Therefore, the discrete-time system (i.e., the smoothing gradient method and equivalent
discretized gradient flow using forward Euler scheme)
precisely simulates the continuous-time system
in terms of the convergence rate.
In other words,
much more effort persuading accurate numerical simulations, such as an implicit numerical methods
and small discretization stepsize, is not expected to improve
the convergence rate.

\subsection{Smoothing Parameter Design based on Continuous-time System}
Besides a stepsize in the discrete-time system has been conventionally designed as a function of the discrete step $k$,
a design exploiting the interpretation that
the smoothing gradient method is a discretized form of the smoothing gradient flow using the forward Euler scheme is
presented herein and provides a non-increasing function $\mu(t)>0$ of the continuous-time $t$.
The continuous-time system, depending on $\mu(t)$, is discretized and numerically integrated
using the stepsize $s_k = 1/(L + \alpha/\mu(t_k))$ in the forward Euler discretization.
The difference equation on $\bsxI{k}$ and $t_k$ is explicitly iterated by
\begin{equation}
	\begin{cases}
		t_{k+1} = t_k + (L + \alpha/\mu(t_k))^{-1},              \\
		\bsxI{k+1} = \bsxI{k} - s_k\FFxgrad{\bsxI{k}}{\mu(t_k)}, \\
		s_k = (L + \alpha/\mu(t_k))^{-1}.
		\elab{proposed}
	\end{cases}
\end{equation}
As illustrated in subsequent numerical examples, the proposed design practically shows preferred convergence behavior.
Furthermore, the following proposition and related corollary support the proposed design framework  theoretically.
\begin{proposition}
	For a given a non-increasing function $\mu(t)>0$ on $t\in[t_0,+\infty)$ satisfying
	$\lim_{t\to+\infty}\mu(t) = 0$,
	a smoothing parameter sequence in the smoothing gradient method \eref{GradientFlowDiscrete} is set as $\mu_k = \mu(t_k)$
	with the stepsize $s_k = 1/L_k = 1/(L + \alpha/\mu(t_k))$.
	Then, $t_k - t_0 = \sum_{\kappa=0}^{k-1} s_\kappa$ diverges to $+\infty$, and $\mu_k \searrow 0$.
\end{proposition}
\begin{proof}
	The boundedness of $t_k$, meaning $t_k \nearrow t_{+\infty} < +\infty$, is assumed,
	and the divergence of $t_k$ is shown by contradiction.
	The bounded $t_k$ results in $\underline{\mu} = \min_{t_0 \leq t \leq t_{+\infty}} \mu(t) > 0$,
	and the following inequality holds:
	\begin{equation}
		t_{k+1} - t_{k} = s_k = \frac{1}{L_k} = \frac{1}{L + \alpha\mu^{-1}_k}
		\geq  \frac{1}{L + \alpha\underline{\mu}^{-1}} > 0.
	\end{equation}
	On recalling $\frac{1}{L + \alpha\underline{\mu}^{-1}}$ is a constant
	and summing up the both side for $k$,
	$t_k$ diverges to $+\infty$; this is a contradiction.
	Therefore, $t_k$ diverges to $+\infty$, and
	$\mu_k = \mu(t_k)$ and $\mu(t) \searrow 0$ results in $\mu_k \searrow 0$.
\end{proof}
The application of the proposition to \corollaryref{SmoothingGradientFlowDiscretefvalConvergence}
and \propositionref{GradientFlowDiscreteConvergence},
the convergence is theoretically guaranteed as follows.
\begin{corollary}
	In the discrete-time system \eref{proposed} for Problem~\eref{Prob},
	$\lim_{k\to+\infty}f(\bsxI{k}) +h(\bsxI{k}) = f(\bsx^*) + h(\bsx^*)$, and
	$\lim_{k\to+\infty} \bsxI{k} = \bsx^*$ if $\sigma > 0$.
\end{corollary}

\section{Numerical Examples}
Based on numerical examples in \cite{Qu22}, a problem
\begin{equation}
	\min_{\bsx\in\R^{n_\bsx}}  \|\bsA\bsx-\bsb\|_2^2 + \|\bsC\bsx-\bsd\|_1
\end{equation}
is taken into consideration here.
The matrices $\bsA\in\R^{n_\bsA \times n_\bsx}$, $\bsC\in\R^{n_\bsC \times n_\bsx}$,
and an optimal solution $\bsx^*\in\R^{n_\bsx}$ are sampled using \texttt{numpy.random.randn},
and the other data are set as $\bsb = \bsA\bsx^*\in\R^{n_\bsA}$, $\bsd = \bsC\bsx^*\in\R^{n_\bsC}$;
therefore, the optimal value is $0$.
Subsequent numerical examples comprise both
non-strongly convex and convex objective functions,
and design parameters are set as $t_0 = 1$ and $\mu_0 = \mu(t_0) = 1$ throughout the numerical examples.

In the $\ell_1$-norm $\|\bsC\bsx-\bsd\|_1$,
an $i$-th row vector of $\bsC_i$ is denoted by $\bsc_i^\top\,(i=1,\dots,n_\bsC)$,
and an absolute value $|\bsc_i^\top\bsx+d_i|$ is approximated by
$\sqrt{\big(\bsc_i^\top\bsx+d_i\big)^2 + \mu^2} - \mu$.
\theoremref{LipschitzSum} shows that $\|\bsC\bsx-\bsd\|_1$
is a $(1/\mu)$-smooth approximation of $\|\bsC\bsx-\bsd\|_1$
with parameters $(\sum_{i=1}^{n_\bsC}\|\bsc_i\|_2^2,n_\bsC)$.

\begin{remark}
	The smoothing parameter of the exponentially decaying form $\mu_k = \mu_0\lambda^k$ numerically behaves like
	a fixed smoothing parameter, persuading an $\varepsilon$-optimal solution,
	and the convergence property depends strongly on the choice of the decaying parameter $\lambda$;
	therefore, the result of the exponentially decaying form is omitted here
	(for further details, see Section 10.8 of \cite{Be17a}).
\end{remark}

\subsection{Strongly Convex Case}
Based on \cite{Qu22}, the parameter of the problem size is set as $n_\bsx = 10, n_\bsA = 20$, and $n_\bsC = 50$.
\begin{figure}[H]
	\begin{minipage}[b]{0.48\columnwidth}
		\centering
		\includegraphics[clip,width=1\textwidth]{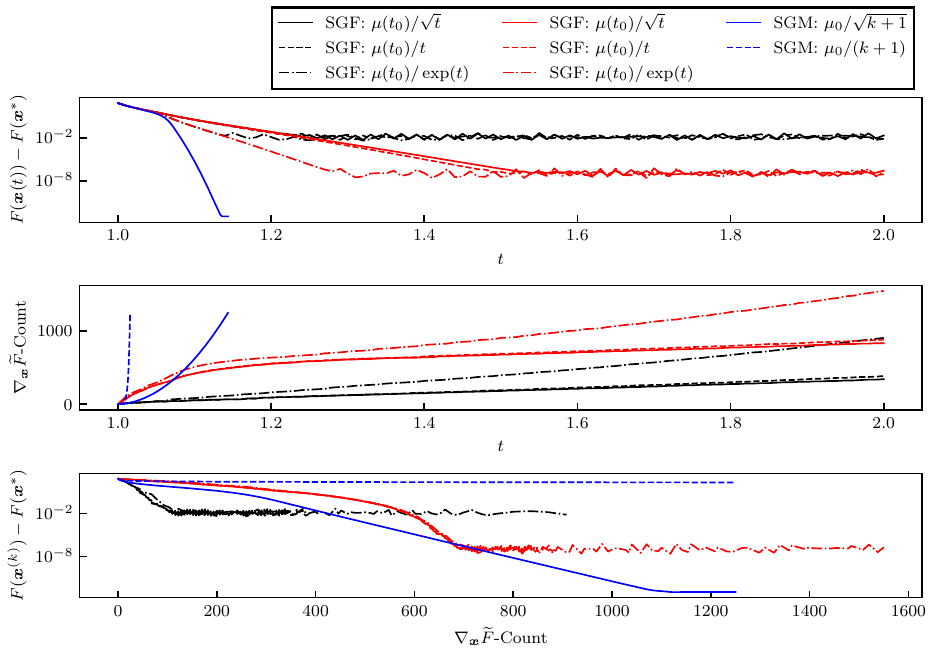}
		\caption{Continuous- and discrete-time state variables
		(the smoothing gradient flow on $t\in[1, 2]$; SGF and
		the smoothing gradient method on $k\in\{0,\dots,1250\}$; SGM, respectively.
		The upper and middle plots are in the continuous-time domain $t$,
		and the bottom is plotted with respect to the count of the $\nabla_\bsx\widetilde{F}$ evaluation.
		SGF is integrated using \texttt{RK45}
		with accuracy parameters
		$(\mathtt{rtol},\mathtt{atol})=(10^{-3},10^{-6})$
		and $(10^{-8},10^{-11})$ for the black and red lines, respectively.). }
		\flab{ODE}
	\end{minipage}
	\begin{minipage}[b]{0.48\columnwidth}
		\centering
		\includegraphics[clip,width=1\textwidth]{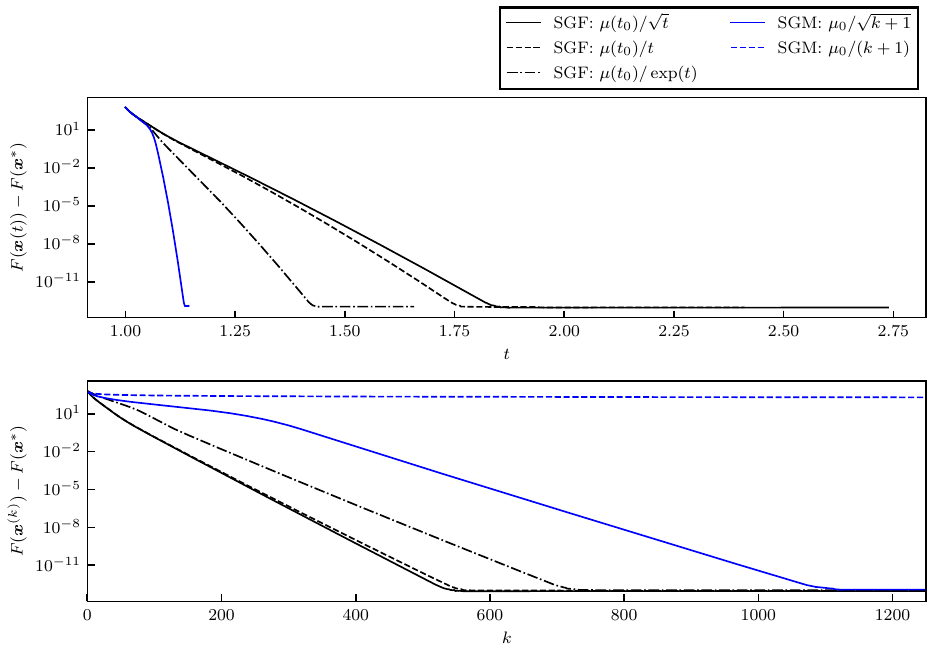}
		\caption{Continuous- and discrete-time state variables
			on $k\in\{0,\dots,1250\}$ (the smoothing gradient flow; SGF and
			the smoothing gradient method; SGM, respectively.
			The upper plot is in the continuous-time domain $t$,
			and the bottom is plotted with respect to the count of the $\nabla_\bsx\widetilde{F}$ evaluation.
			SGF is integrated using the proposed forward Euler scheme).
			\newline
			\newline}
		\flab{Euler}
	\end{minipage}
\end{figure}
First, the numerical integration using an ODE solver \texttt{RK45} of a package \verb+scipy.integrate.solve_ivp+
(hereafter, this scheme is called ODE solver) is performed,
and its practical computation efficiency is examined.
The result of the ODE solver is depicted in \fref{ODE}.

In the computation with the ODE solver,
the terminal time $T$ for the numerical integration is gradually increased,
and the resulting value of $\bsx(T)$ and the required number of the $\nabla_\bsx\widetilde{F}$ evaluation are observed.
The $\nabla_\bsx\widetilde{F}$ evaluation is counted using a global variable without modifying aforementioned \texttt{RK45}.
As shown in \corollaryref{GradientFlowDiscreteNonStrongly},
the smoothing gradient method, labeled as ``SGM'', with $\mu_k = \mu_0/\sqrt{k+1}$
performs better convergence behavior.
In the numerical integration with ODE solver, labeled as ``SGF'' with red lines,
the number of the $\nabla_\bsx\widetilde{F}$ evaluation is suppressed compared with SGM;
however, the improvement of the objective function value is terminated
at the early stage.
In a practical situation solving an optimization problem,
the aforementioned parameter tuning requires much effort
and the $\nabla_\bsx\widetilde{F}$ evaluation for the tuning procedure itself;
therefore, further discussion on the usage of the ODE solver is out of this paper's scope and omitted.

To explore the smoothing parameter design in the discrete-time domain (and possible related application to optimization problems),
the comparison on the smoothing parameter design is illustrated in \fref{Euler}.
Besides \corollaryref{GradientFlowDiscreteNonStrongly} indicates that
the form of $\mu_k = \mu_0/\sqrt{k+1}$ achieves better convergence,
the proposed forward Euler scheme for SGF numerically shows preferred convergence property.

\subsection{Non-Strongly Convex Case}
A numerical experiment with a setting of $n_\bsx = 10, n_\bsA = 2, n_\bsC = 5$ is illustrated here.
The objective function is convex but
the coefficient of the quadratic cost of $\bsx$ (i.e., $\bsA^\top\bsA$)
has its rank $n_\bsA = 2$ or less because $n_\bsA = 2 < n_\bsx = 10$, meaning
it has 8 or more zero eigenvalues; therefore the objective function is non-strongly convex.

The comparison with ODE solver is illustrated in \fref{ODE2}.
The result with the proposed forward Euler scheme for SGF numerically
is also depicted in \fref{Euler2}. Similar to the aforementioned strongly convex case,
the practical advantage of the proposed forward Euler scheme is observed.
\begin{figure}[H]
	\begin{minipage}[b]{0.48\columnwidth}
		\centering
		\includegraphics[clip,width=1\textwidth]{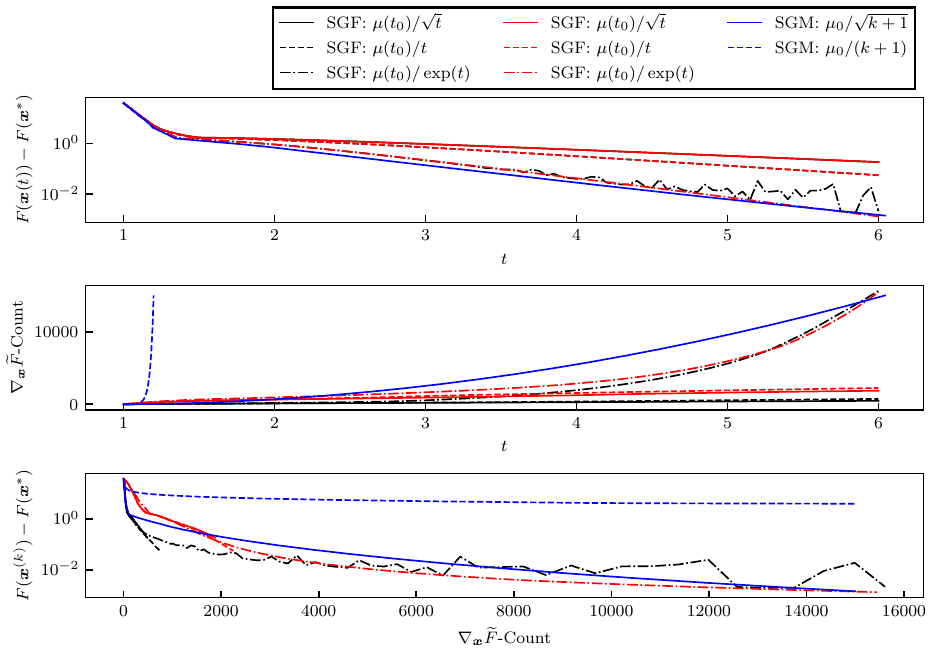}
		\caption{Continuous- and discrete-time state variables
		(the smoothing gradient flow on $t\in[1, 6]$; SGF
		and the smoothing gradient method on $k\in\{0,\dots,15000\}$; SGM, respectively.
		The upper plot is in the continuous-time domain $t$,
		and the bottom is plotted with respect to the count of the $\nabla_\bsx\widetilde{F}$ evaluation.
		SGF is integrated  using \texttt{RK45}
		with accuracy parameters
		$(\mathtt{rtol},\mathtt{atol})=(10^{-3},10^{-6})$
		and $(10^{-8},10^{-11})$ for the black and red lines, respectively.). }
		\flab{ODE2}
	\end{minipage}
	\begin{minipage}[b]{0.48\columnwidth}
		\centering
		\includegraphics[clip,width=1\textwidth]{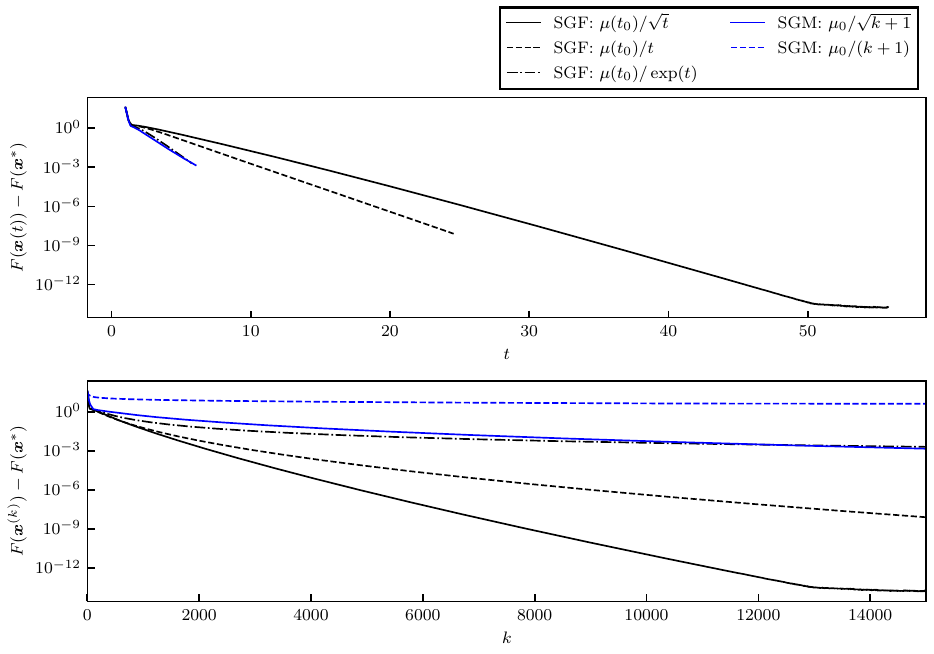}
		\caption{Continuous- and discrete-time state variables
			on $k\in\{0,\dots,15000\}$ (the smoothing gradient method; SGM and
			the smoothing gradient flow; SGF, respectively. The upper plot is in the continuous-time domain $t$,
			and the bottom is plotted with respect to the count of the $\nabla_\bsx\widetilde{F}$ evaluation.
			SGF is integrated using the proposed forward Euler scheme.).
			\newline
			\newline}
		\flab{Euler2}
	\end{minipage}
\end{figure}

\section{Concluding Remarks}
This paper presented the convergence analysis of
the smoothing gradient flow (a continuous-time ODE) and the smoothing gradient method (a discrete-time difference equation)
using the Lyapunov functions.
This paper focused on the discrete-time system obtained using the forward Euler scheme,
which is equivalent to the smoothing gradient method,
because the convergence property of the continuous-time system does not hold in
the approximate numerical integration and requires additional analysis in the discrete-time domain.
The new smoothing parameter design is proposed based on the forward-Euler-based numerical integration of the smoothing gradient flow
and numerically showed its advantage compared with the existing smoothing parameter sequences.

\section{Acknowledgment}
This work was supported by JSPS KAKENHI Grant Number JP22K14279.


\begin{thebibliography}{10}
	\bibitem{Pa14}
	N.~Parikh and S.~Boyd, ``Proximal algorithms,'' \emph{Found. Trends Optim.},
	vol.~1, no.~3, pp. 127--239, Jan. 2014.

	\bibitem{Be17a}
	A.~Beck, \emph{First-Order Methods in Optimization}.
	Philadelphia, PA, USA: Society for Industrial and Applied
	Mathematics, 2017.

	\bibitem{Ch12b}
	X.~Chen, ``Smoothing methods for nonsmooth, nonconvex minimization,''
	\emph{Math. Program.}, vol. 134, no.~1, pp. 71--99, Jun. 2012.

	\bibitem{Bi20}
	W.~Bian and X.~Chen, ``A smoothing proximal gradient algorithm for nonsmooth
	convex regression with cardinality penalty,'' \emph{SIAM J. Numer. Anal.},
	vol.~58, no.~1, pp. 858--883, Jan. 2020.

	\bibitem{Wu23}
	F.~Wu and W.~Bian, ``Smoothing accelerated proximal gradient method with fast
	convergence rate for nonsmooth convex optimization beyond
	differentiability,'' \emph{J. Optim. Theory Appl.}, vol. 197, no.~2, pp.
	539--572, 2023.

	\bibitem{Br89}
	A.~A. Brown and M.~C. Bartholomew-Biggs, ``Some effective methods for
	unconstrained optimization based on the solution of systems of ordinary
	differential equations,'' \emph{J. Optim. Theory Appl.}, vol.~62, no.~2, pp.
	211--224, Aug. 1989.

	\bibitem{Su16}
	W.~Su, S.~Boyd, and E.~J. Cand{{\`e}}s, ``A differential equation for modeling
		{Nesterov}'s accelerated gradient method: Theory and insights,'' \emph{J.
		Mach. Learn. Res.}, vol.~17, no. 153, pp. 1--43, Sep. 2016.

	\bibitem{Ne14}
	Y.~Nesterov, \emph{Introductory Lectures on Convex Optimization: A Basic
		Course}, Boston, MA, USA:
	Springer, 2014.

	\bibitem{Qu22}
	X.~Qu and W.~Bian, ``Fast inertial dynamic algorithm with smoothing method for
	nonsmooth convex optimization,'' \emph{Comput. Optim. Appl.}, vol.~83, no.~1,
	pp. 287--317, Jul. 2022.

	\bibitem{Us22}
	K.~Ushiyama, S.~Sato, and T.~Matsuo, ``Essential convergence rate of ordinary
	differential equations appearing in optimization,'' \emph{JSIAM Lett.},
	vol.~14, pp. 119--122, Sep. 2022.

	\bibitem{Do80}
	J.~Dormand and P.~Prince, ``A family of embedded {Runge}--{Kutta} formulae,''
	\emph{J. Comput. Appl. Math.}, vol.~6, no.~1, pp. 19--26, Mar. 1980.

	\bibitem{Pr92}
	W.~H. Press, S.~A. Teukolsky, W.~T. Vetterling, and B.~P. Flannery,
	\emph{Numerical Recipes in C: The Art of Scientific Computing}, 2nd~ed.
	New York, NY, USA: Cambridge University
	Press, 1992.

	\bibitem{Wi18}
	A.~Wilson, ``{Lyapunov} arguments in optimization,'' phdthesis, University of
	California, Berkeley, 2018.

	\bibitem{Ha91}
	A.~Haraux, \emph{Syst{\`e}mes dynamiques dissipatifs et applications}, ser.
	Recherches en math{\'e}matiques appliqu{\'e}es.
	Paris, France: Masson, 1991.

	\bibitem{Bi20a}
	W.~Bian, ``Smoothing accelerated algorithm for constrained nonsmooth convex
	optimization problems (in {Chinese}),'' \emph{Sci. Sin. Math.}, vol.~50,
	no.~12, pp. 1651--1666, 2020.

	\bibitem{Fu13}
	O.~Furdui, \emph{Limits, Series, and Fractional Part Integrals: Problems in
		Mathematical Analysis}.
	New York, NY,
	USA: Springer New York, 2013.

\end{thebibliography}
\end{document}